\documentclass[11pt]{article}
\usepackage[]{geometry}
\usepackage{graphicx}
\usepackage{caption}
\usepackage{subcaption}
\usepackage{hyperref}
\usepackage{amsmath}
\usepackage{amsthm}
\usepackage{amssymb}
\usepackage{authblk}
\usepackage{amsmath}
\usepackage{bm}
\usepackage{xcolor}
\usepackage{float}
\usepackage{booktabs}
\usepackage{enumitem}
\usepackage{yhmath}
\usepackage{mathdots}
\usepackage{MnSymbol}

\usepackage{pgfplots}
\usepgfplotslibrary{fillbetween}

\everymath{\displaystyle}
\allowdisplaybreaks

\newtheorem{theorem}{Theorem}
\newtheorem{definition}[theorem]{Definition}
\newtheorem{proposition}[theorem]{Proposition}
\newtheorem{remark}[theorem]{Remark}
\newtheorem{hyp}[theorem]{Hypothesis}
\newtheorem{remarks}[theorem]{Remarks}

\newtheorem{lemma}[theorem]{Lemma}

\newcommand{\dd}{\mathrm{d}}

\providecommand{\keywords}[1]
{ \small \textbf{\textit{Keywords---}}#1 }

\title{Uniqueness and stability in bottom detection through surface measurements of water waves}
\author[$\;$]{}

\affil[$\;$]{ {\large Noureddine LAMSAHEL}}
\affil[$\;$]{ {\scriptsize  Mohammed VI Polytechnic University, The UM6P Vanguard Center, Benguerir 43150, Lot 660, Hay Moulay Rachid, Morocco.}}
\affil[$\;$]{ {\scriptsize   Universit\'e du Littoral C\^ote d'Opale, Laboratoire de Math\'ematiques Pures et Appliqu\'ees J. Liouville, B.P. 699, F-62228 Calais, France}.}
\affil[$\;$]{ {\scriptsize Emails: noureddine.lamsahel@um6p.ma, noureddine.lamsahel@etu.univ-littoral.fr  }}
\affil[$\;$]{ {\large }}
\affil[$\;$]{ {\large Lionel ROSIER}}
\affil[$\;$]{ {\scriptsize   Universit\'e du Littoral C\^ote d'Opale, Laboratoire de Math\'ematiques Pures et Appliqu\'ees J. Liouville, B.P. 699, F-62228 Calais, France}.}
\affil[$\;$]{ {\scriptsize Email:  lionel.rosier@univ-littoral.fr}}
\affil[$\;$]{ {\large }}

\date{}

\begin{document}

\maketitle

\begin{abstract}
This paper investigates the geometric inverse problem of recovering the bottom shape from surface measurements of water waves.  Using the general water-waves system on a bounded subdomain of the fluid domain, we address this inverse problem, focusing on the identifiability and the stability issues. We establish uniqueness and derive logarithmic stability estimates in the determination of  the bathymetry on any fixed smooth, bounded, open domain
 ${\mathcal O}\subset {\mathbb R} ^d$, $d=1,2$, from the knowledge of the free surface, its first time derivative,   and the trace of the velocity potential on the free surface, at a given instant $t_0$ within $\mathcal O $,   together with  the knowledge of the bottom along $\partial \mathcal O$.  No further assumptions are required for uniqueness. For stability, we impose only a \textit{local fatness} condition on the region between the bottom profiles, allowing us to adapt the size estimates method.
\end{abstract}

\noindent\keywords{Water waves, Elliptic systems,  Geometric inverse problems, Identifiability,  Log-Log stability, Log stability, Size estimates, Bottom.}

\section{Introduction}\label{intropaper}
Knowledge of river and ocean bathymetry is indispensable for many engineering applications, including the design of offshore platforms and harbors \cite{brocchini2013reasoned}, the study of channel-flow hydrodynamics \cite{cunge1980practical,marks2000integration}, and the approximations of tsunami run-up \cite{grilli1994shoaling,synolakis1987runup,synolakis1991greens}. Estimating this geometry is a challenging problem that often relies on \textit{in situ} measurements \cite{irish1999scanning,lecours2016review,smith2004conventional}. However, these techniques are time-consuming and costly, especially for large aquatic domains \cite{hilldale2008assessing,lecours2016review,sagawa2019satellite}. Due to these limitations, an inverse-problem approach emerges to infer the underwater geometry by selecting appropriate governing equations for the flow dynamics and using measurements taken at the water surface. (See the review paper \cite{sellier2016inverse} and the references therein.) Crucially, surface measurements, such as the wave profile \cite{fontelos2017bottom,nicholls2008joint,nicholls2009detection,vasan2013inverse} or the surface velocity potential \cite{adrian2011piv,fontelos2017bottom}, are more accessible than bottom topography \cite{hilldale2008assessing,smart2009river}.

In this work, we formulate the detection of the bottom profile as an inverse problem in which the flow dynamics is governed by the so-called \textit{general water–waves} system \cite{lannes2013water}. This system describes the motion of a layer of an inviscid, incompressible, and irrotational fluid limited below by a solid bottom and above by a free surface, with both boundaries assumed to be parametrized. We define this spatial region as 
$$\Omega_t=\{(X,y)\in\mathbb{R}^d\times\mathbb{R},\; b(X)<y<\zeta(t,X)\},\quad d\in \{1,2\}.$$
From \cite{lannes2013water}, the system reads
\begin{equation}\label{pd0}
	\left\{
	\begin{array}{ll}
		\Delta\phi =0,&\quad \textrm{ in } \Omega_t,\quad \\
		\partial_t\zeta+\nabla_X\zeta\cdot\nabla_X\phi=\partial_y\phi,&\quad \textrm{ for } y=\zeta, \quad \\
		\partial_t\phi+\frac{1}{2}\left(\left|\nabla_X\phi\right|^2+\left(\partial_y\phi\right)^2\right)+g\zeta=0,&\quad \textrm{ for } y=\zeta, \quad \\
		\nabla_Xb\cdot\nabla_X\phi-\partial_y\phi=0,&\quad \textrm{ for } y=b.\quad 
	\end{array}
	\right.
\end{equation}
Here, $g$ stands for gravitational acceleration and $\phi$ is the velocity potential related to fluid velocity $U$ through the relation $U=\nabla\phi$.
Throughout this paper, we set 
$$\nabla=\begin{pmatrix}
\nabla_X \\[2pt]
\partial_y 
\end{pmatrix}\quad\text{ and }\quad \Delta=\Delta_X+\partial_y^2.$$
Moreover, we adopt the hypotheses outlined in \cite{lannes2013water}.
Additionally, we exclude the case of water at rest, that is, we assume that the velocity $U$ is not zero (a.e.).  

The system (\ref{pd0}) can be reformulated to exclusively characterize the behavior on the free surface \cite{zakharov1968stability}. Following \cite{craig1993numerical} and using  the Dirichlet to Neumann operator (DNO), we obtain the following system: 
\begin{equation}\label{pd1}
\left\{
\begin{array}{ll}
\partial_t\zeta =G(\zeta,b)\psi ,\quad \\
		\partial_t\psi=-g\zeta-\frac{1}{2}\left|\nabla_X\psi\right|^2+\dfrac{\left(     G(\zeta,b)\psi+\nabla_X\zeta\cdot\nabla_X\psi\right)^2}{2\left(1+|\nabla_X\zeta|^2\right)}, \quad \\
		\zeta(0,X)=\zeta_0(X),\;\psi(0,X)=\psi_0(X),\quad\\
\end{array}
\right.
\end{equation}
where $\psi(t,X)=\phi(t,X,\zeta(t,X))$ is the trace of the velocity potential at the free surface $\zeta$, and  $G$ is the so-called \textit{Dirichlet to Neumann operator}, defined by 
\begin{equation}\label{Gthe DNO}
	\psi\longrightarrow G(\zeta,b)\psi=\sqrt{1+|\nabla_X\zeta|^2}\;\partial_n\phi_{|_{y=\zeta}},
\end{equation}
where $\phi$ is the solution to the elliptic system 
\begin{equation}\label{pd2}
	\left\{
	\begin{array}{ll}
		\Delta\phi =0,\hspace{1.01cm} \textrm{ in }  \Omega_t,\\
		\phi_{|_{y=\zeta}}=\psi,\;\partial_n\phi_{|_{y=b}} =0.\quad\\
	\end{array}
	\right.
\end{equation}
For the well-posedness of (\ref{pd1})-(\ref{pd2}) in different Sobolev spaces, we refer the reader to \cite{alazard2016cauchy,lannes2013water}.

The inverse problem associated with \eqref{pd1}-\eqref{pd2} consists of recovering the bottom profile $b$ from measurements taken at the water surface $\zeta$. In \cite{nicholls2008joint}, the problem is treated numerically. However, the approach requires standing wave profiles and thus has limited applicability in real-world settings \cite{nicholls2009detection}. In \cite{vasan2013inverse}, the standing wave assumption is removed by imposing a periodic velocity potential and small-amplitude waves. Recently, in \cite{fontelos2017bottom}, the authors established mathematical uniqueness for this inverse problem without restrictions on velocity or wave type. Precisely, given two different bottoms $b$ and $b_0$ in $ C^5(\mathbb{R}^d)$  and their corresponding surface elevations  $\zeta,\, \zeta_0$ and velocity potentials  $\phi_,\, \phi_0$, the authors proved that, if for some nonempty open set $S$ in $\mathbb{R}^d$ and for some  time $t_0$ we have  $$\zeta(t_0,X)=\zeta_0(t_0,X)\neq 0,\;\; \psi(t_0,X)=\psi_0(t_0,X)\neq 0\;\;\text{and }\;\; \partial_t\zeta(t_0,X)=\partial_t\zeta_0(t_0,X),\quad \forall X\in S,$$
then 
$$b(X)=b_0(X),\quad \forall X\in \mathbb{R}^d.$$
However, the preceding result relies on an unbounded fluid-domain setting (or, equivalently, on a bounded fluid domain with solid, impermeable vertical walls) and on a smooth bathymetry (see \cite[Section 2.2]{fontelos2017bottom}). In \cite{lecaros2020stability}, the stability is addressed. However, due to geometric complexity, the authors restrict the stability analysis to the case $\zeta(t_0,X) = \zeta_0(t_0,X)$. Moreover, in the region between the bottoms, the study assumes that the set of intersection points between $b$ and $b_0$ is (at most) finite. Furthermore, the results require a bounded $C^{1,1}$ domain with solid (impermeable) walls.

To treat both bounded and unbounded fluid domains, we consider the inverse problem of detecting the seafloor over an arbitrary bounded open set $\mathcal O \subset \mathbb{R}^d$ with boundary $C^1$. The choice of a bounded truncation of the fluid domain is motivated by our interest in establishing a stability estimate using results for elliptic systems that are available (to the best of our knowledge) only for bounded domains. Moreover, this truncation is essential for the numerical implementation that we intend to develop in future work; see Section \ref{sec:conclusions}. Specifically, adopting the surface measurements of \cite{fontelos2017bottom, lecaros2020stability}, we consider the $(d+1)$-dimensional water–waves system \eqref{pd1}–\eqref{pd2} for $d\in\{1,2\}$ on an unbounded fluid domain $\Omega_t$, and we study identifiability and stability for recovering an arbitrary bounded portion of the bottom $b$ using measurements collected only above it. The unbounded fluid domain setting ensures that the well-posedness difficulties associated with \eqref{pd1}–\eqref{pd2} \cite{benjamin1979gravity} are avoided. 
Moreover, in contrast to \cite{lecaros2020stability}, we do not confine our analysis to the situation in which the normal derivative of the velocity potential vanishes on vertical walls, since in our setting such walls are not necessarily impermeable.

The primary objective of this paper is to prove stability and uniqueness under minimal assumptions. Specifically, we aim to prove identifiability for the geometric inverse problem on the truncated domain and show that we can recover the bottom profile without using inlet or outlet data for the velocity potential. Moreover, we aim to derive stability estimates without several conservative assumptions used in \cite{lecaros2020stability}. More precisely, we do not assume:
\begin{enumerate}[label=(\roman*)]
  \item a homogeneous Neumann condition for the velocity potential on the vertical walls (e.g., on $\Gamma_b^{\zeta}$ in Figure~\ref{fig:myplot1});
  \item  that the two free surfaces coincide at some time $t_0$ $(\zeta(t_0,\cdot)=\zeta_0(t_0,\cdot))$;
  \item that the two bottoms intersect only at finitely many points;
  \item that the truncated domain is of class $C^{1,1}$. (Instead, we assume only that the domain has a Lipschitz boundary, which is the optimal regularity in the presence of vertical boundaries, even when the free surface and the bottom are smooth.);
  \item that the region between the bottoms satisfies the fatness condition. Rather, we adopt a relaxed local fatness condition that allows for infinitely many intersections between the bottoms; see Hypothesis \ref{hyp-of-paper}. In particular, this illustrates the possibility of using the size estimates method to detect a countably infinite disjoint union of subdomains.
\end{enumerate}

\subsection{Main results}
Consider an arbitrary bounded open set $\mathcal O \subset \mathbb R^d$ ($d\in\{1,2\}$) with boundary $C^1$, corresponding to the bottom region we intend to detect
(see Figure \ref{fig:myplot1}). Let us fix an instant of measurement $t_0$. We define the following operator
\begin{equation}\label{operatorofbottom}
	\Lambda_{t_0}:\;\; b \longmapsto\left( \left(\zeta, \partial_n\phi_{|_{y=\zeta}}, \psi\right)|_{\{t_0\}\times \mathcal O} ,b_{\vert \partial \mathcal O}\right),
\end{equation}
where $\phi\in H^2(\Omega_{t_0})$  is the solution to the elliptic system (\ref{pd2}). Then our main contributions can be summarized as follows:

\begin{enumerate}
	\item  The operator $\Lambda_{t_0}$ is locally one-to-one; that is, the bottom geometry on $\mathcal O$ can be identified from the given measurements.  (See Theorem \ref{uniqtheorem} below.) This result extends \cite{fontelos2017bottom} to the case of a truncated fluid domain and a $C^1$ bottom profile. 
	
	\item We establish $L^{1}$ log--log and logarithmic stability estimates for the inverse problem under consideration,  dispensing with the hypotheses of \cite{lecaros2020stability} and assuming only the local fatness Hypothesis~\ref{hyp-of-paper}.  (See Theorem~\ref{stabilitytheorem} and Remark~\ref{possiblelogstability}).

	\item When the fluid domain is bounded (solid walls), as in \cite{lecaros2020stability}, or when the bottom profile is flat except on a compact set, our analysis suggests that the knowledge of $b(\partial \mathcal O)$ is not needed.
\end{enumerate}

\subsection{Outline of the paper}
In Section \ref{sec:backand-Problm}, we provide the necessary notations, definitions, and preliminary results crucial for our analysis. We start by formally presenting the inverse problem along with the paper notations. Afterwards, we collect from the literature some results on elliptic problems in Lipschitz domains. In Section \ref{sec:uniqueness},  we address uniqueness. After establishing a pivotal result (Lemma \ref{Lemma1useful}), we employ the unique continuation and the Lipschitz propagation of smallness for elliptic problems to prove the identifiability result for bottom detection (Theorem \ref{uniqtheorem}).
In Section \ref{sec:stability},  we prove the main result of the paper (Theorem~\ref{stabilitytheorem}) providing stability estimates. To do it, we 
apply classical Sobolev embeddings together with some results from Section \ref{sec:uniqueness} to obtain an upper bound for the inter-bottom energy, expressed in terms of the differences between the selected surface measurements. Subsequently, we use a standard technique from the size estimates method to bound this energy from below by the measure of the region between the bottoms.
To enable this technique, we relax the classical fatness condition to permit infinitely many intersections between the bottoms; see Hypothesis \ref{hyp-of-paper}. For ease of reading, all the proofs of this section that require direct (but cumbersome) computations are shifted to Appendix~\ref{A.1}--\ref{A.2}. Finally,
Section \ref{sec:conclusions} summarizes the results, provides some further comments, and discusses future work.

\section{Problem setting and preliminaries}\label{sec:backand-Problm}
In this section, we present the notation, definitions, and results needed for the analysis in the subsequent sections. We first formulate the paper inverse problem, and then review well-established results for elliptic problems on Lipschitz domains.

\subsection{The geometric inverse problem}\label{sec:introinverseproblem}
We consider the recovery of a bounded part of the seafloor from surface measurements taken above it. Let \(\mathcal O\) be an arbitrary bounded open subset of \(\mathbb{R}^d\) ($d\in \{1,2\}$) such that \(\partial\mathcal O \in C^{1}\). Throughout the paper, we refer to $\Omega_f^g$ as the domain defined by
\begin{equation}\label{domainrefdef}
    \Omega_f^g=\left\{(X,y)\in \mathcal O \times\mathbb{R},\; f(X)<y<g(X)\right\},
\end{equation}
where $f,g\in C^{0,1}(\overline{\mathcal O})$ are such that  $f<g$ on $\overline{\mathcal O }$. Additionally, we separate the boundary of $\Omega_f^g$ into three parts as follows
$$\Gamma^g=\left\{ (X,y)\in \mathcal O \times\mathbb{R},\;y=g(X)\right\},\;\Gamma^f=\left\{ (X,y)\in \mathcal O \times\mathbb{R},\;y=f(X)\right\},$$
and 
$$\Gamma^{g}_f=  \{ (X,y)\in \partial \mathcal O \times\mathbb{R},\ f(X)<y<g(X) \} .\;$$
From  (\ref{pd1})-(\ref{pd2}) and using the above notations, we assume that at a given time $t_0$ we have access to the surface measurements $(\zeta, \partial_n\phi_{|_{\Gamma^{\zeta}}}, \psi)|_{\{t_0\}\times \mathcal O }$ along $\Gamma^{\zeta}$  within  $\mathcal O $. Moreover, we assume that the values of the bottom at the boundary points are known, namely $b(\partial \mathcal O )$ (see Figure \ref{fig:myplot1}, where $b(\partial \mathcal O )=\{b(a_1),b(a_2)\}$).

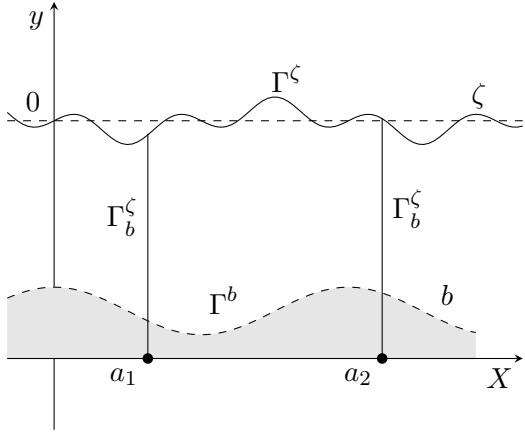
\begin{figure}
	\centering
	\begin{tikzpicture}
		\begin{axis}[
			axis lines = middle,
			xlabel = $X$,
			ylabel = $y$,
			xtick =  \empty,
			ytick = \empty,
			xmin = -1,
			xmax = 10,
			ymin = -3, 
			ymax = 15, 
			no markers,
			grid = none,
			every axis x label/.style={at={(current axis.right of origin)},anchor=north east},
			every axis y label/.style={at={(current axis.above origin)},anchor=north east},
			]
			
			\addplot[black, domain=-1:10, samples=100] {sin(deg(x))*cos(deg(x*2))+10} node[pos=0.9, anchor=south west] {$\zeta $} node[pos=0.5, anchor=south west] {$\Gamma^{\zeta} $};
			\addplot[black, domain=-1:9, samples=100, name path=A] {0} ;
			
			\addplot[dashed, domain=-1:9, samples=100, name path=B] {cos(deg(x))+2} node[pos=0.9, anchor=south west] {$b $} node[pos=0.5, anchor=south east] {$\Gamma^{b} $};
			
			\draw[black] (axis cs:2,0) -- (axis cs:2,9.45) node[pos=0.5, anchor=south east] {$\Gamma^\zeta_b $} ;
			\draw[black] (axis cs:7,0) -- (axis cs:7,10.1) node[pos=0.5, anchor=south west] {$\Gamma^\zeta_b  $};
			
			\fill (axis cs:2,0) circle[radius=2pt] node[below left] {$a_1$};
			\fill (axis cs:7,0) circle[radius=2pt] node[below left] {$a_2$};
			
			\draw[dashed] (axis cs:-1,10) -- (axis cs:10,10)  node[pos=0.05, above] {0};
			
			
			
			\addplot[gray!20] fill between[of=A and B];
			
		\end{axis}
	\end{tikzpicture}
	\caption{ Scheme for the inverse problem when $\mathcal O=(a_1,a_2)\subset \mathbb R$ and $d=1$.}
	\label{fig:myplot1}
\end{figure}

As pointed out in the Introduction, this work aims to demonstrate the local injectivity and stability of the operator $\Lambda _{t_0}$ given by \eqref{operatorofbottom}. It is important to mentioned that using the first equation in (\ref{pd1}), we can replace the measurements $\left(\zeta, \partial_n\phi_{|_{\Gamma^{\zeta}}}, \psi\right)|_{\{t_0\}\times \mathcal O }$ taken at an instant $t_0$ by $\left(\zeta|_{[0,T ) \times \mathcal O },\psi|_{\{t_0\}\times \mathcal O }\right)$, where $t_0$ is in the interval $(0,T)$. On the other hand, it is clear that in comparison to the works \cite{fontelos2017bottom,lecaros2020stability}, we have included additional measurements, namely $b(X)$ for $X\in \partial \mathcal O$. However, in the upcoming sections, it will be clear that if we consider a bounded fluid domain as in  \cite{lecaros2020stability}, these supplementary measurements are not required. Moreover, working on a bounded truncated domain, where the boundary information $b|_{\partial\mathcal O}$ is used instead of inlet and outlet velocity data for the inverse problem, is numerically more practical than considering an unbounded fluid setting in \cite{fontelos2017bottom}. Furthermore, we emphasize that boundary information on the bathymetry is commonly used as an input for this geometric inverse problem when using the shallow water equations (see e.g. \cite{angel2024bathymetry,gessese2011reconstruction,meandCarole}.)

Without loss of generality and to ensure $H^2$-regularity for the velocity potential, we suppose that the fluid domain is unbounded in the horizontal direction $\mathbb{R}^{d}$, where $d=1,2.$

Fix $d_0>\frac{d}{2}$, and let $b,b_0,\zeta,\zeta_0 \in H^{d_0+1}(\mathbb{R}^d)$ and $\psi,\psi_0 \in \dot{H}^{3/2}(\mathbb{R}^d)$, where
$$\dot{H}^{3/2}(\mathbb{R} ^d)=\left\{  \psi\in L^2_{loc}(\mathbb{R} ^d),\quad\nabla_X\psi\in H^{1/2}(\mathbb{R} ^d)^{d} \right\}.$$
We assume that the depth of the water is always bounded from below by a positive constant (see $(\text{H}9)$ in \cite{lannes2013water}); that is,  there exists a constant $H_0>0$ such that 
\begin{equation}\label{conditionofliquid}
	\zeta_0(X)-b_0(X)\geq H_0\;\;\text{and }\;\; \zeta(X)-b(X)\geq H_0,\quad \forall X\in \mathbb{R} ^d. 
\end{equation}
It was proved in \cite{lannes2013water} (see also \cite[Lemma 2.5]{fontelos2017bottom}) that system (\ref{pd2}), for the chosen time $t_0$, has a unique solution $\phi$ in $H^2(\Omega _{t_0})$. Let $\theta :=\phi _{\vert  \Gamma ^\zeta _b }$.  Then  the following system

\begin{equation}\label{pd3}
	\left\{
	\begin{array}{rrrrr}
		\vspace{0.1cm}
		\Delta\phi &=&0,&\quad \text{on}\; \; \Omega_{b}^{\zeta},\\
		\vspace{0.1cm}
		\phi &=&\psi,&\quad \text{in}\;\;  \Gamma^{\zeta},\\
		\partial_n\phi &=&0,&\quad \text{in}\;\;\Gamma^{b},\\
		\phi &=&\theta ,&\quad \text{in}\;\;  \Gamma^{\zeta}_b,\\
	\end{array}
	\right.
\end{equation}
\vspace{0.2cm}
has a unique solution, also denoted by $\phi$, in $H^2(\Omega_b^{\zeta}).$

Replacing $b$, $\zeta$, $\psi$ and $\theta $ by  $b_0$, $\zeta_0,$ $\psi_0$, and $\theta _0 :={\phi _0}_{\vert \Gamma^{\zeta_0} _{b_0}}$, respectively, and following similar arguments, the system

\begin{equation}\label{pd4}
	\left\{
	\begin{array}{rrrrr}
		\vspace{0.1cm}
		\Delta\phi_0 &=&0,&\quad \text{on}\; \; \Omega_{b_0}^{\zeta_0},\\
		\vspace{0.1cm}
		\phi_0 &=&\psi_0,&\quad \text{in}\;\;  \Gamma^{\zeta_0},\\
		\partial_n\phi_0 &=&0,&\quad \text{in}\;\;\Gamma^{b_0},\\
		\phi _0 &=&\theta _0 ,&\quad \text{in}\;\;  \Gamma^{\zeta_0}_{b_0},\\
	\end{array}
	\right.
\end{equation}
has a  unique solution $\phi_0 \in H^2(\Omega_{b_0}^{\zeta_0}).$

We conclude this subsection with the following remarks.
\begin{remark}\label{notationt0}
In condition \eqref{conditionofliquid} and in systems \eqref{pd3}–\eqref{pd4}, the quantities depend on the selected measurement time $t_0$. For simplicity, we omit this dependence. Throughout the paper,  all terms refer to the instant $t_0$,  e.g.
$$\zeta(X):=\zeta(t_0,X),\;\; \psi(X):=\psi(t_0,X).$$

\end{remark}
\begin{remark}\label{Assumption of paper}
    By Sobolev embedding, since $b, b_0, \zeta,$ and $\zeta_0 \in H^{d_0+1}(\mathbb{R}^d)$ with $d_0>\frac{d}{2}$ and $d\in\{1,2\}$, it follows that $b, b_0, \zeta,$ and $\zeta_0 \in C^1(\overline{\mathcal{O}})$. Throughout this paper, we work in the following  general setting:
\begin{enumerate}
  \item $b, b_0, \zeta,$ and $\zeta_0 \in C^1(\overline{\mathcal{O}})$;
  \item $\phi \in H^2(\Omega_b^{\zeta})$ and $\phi_0 \in  H^2(\Omega_{b_0}^{\zeta_0})$ are solutions to \eqref{pd3} and \eqref{pd4}, respectively;
  \item Condition~\eqref{conditionofliquid} holds.
\end{enumerate}

\end{remark}

\subsection{Preliminaries on Elliptic problems}
In this part, we gather from the literature some well-established results on elliptic problems, including Lipschitz propagation of smallness, log-log stability, and logarithmic stability. We begin with the following definition.

\begin{definition}[\cite{alessandrini2003size}, Lipschitz regularity]\label{lipsdef}
	Let $\Omega$ be a bounded open set  in $\mathbb{R}^{d+1}$, $d\ge 1$. We  say that $\partial \Omega$ is in the Lipschitz class with constants $r_0,\;M_0$ if, for any $(X_0,y_0)\in\partial \Omega$, there exists a rigid transformation (i.e. an isometry) of coordinates under which we have $(X_0,y_0)=0$ and 
	$$\Omega\cap B^{d+1}_{r_0}(0)=\{ (X,y) \in B^{d+1} _{r_0}(0),\;\;y>\gamma(X)\},$$
	where $B^{d+1}_{r_0}(0) : = \{ (X,y)\in {\mathbb R} ^{d+1},\  ||(X,y)||<r_0\}$ and $\gamma$ is a Lipschitz continuous function on $B_{r_0}^{d}(0)$, satisfying
	$$\gamma(0)=0\quad \text{and}\quad ||\gamma||_{C^{0,1}(B^{d}_{r_0} (0)) }\leq M_0 r_0.$$
	
\end{definition}

\begin{remarks}\label{remarks-on-the-Lipshitz-domain}$\quad$\\
	\begin{itemize}
		\item Observing the above Definition \ref{lipsdef}, it becomes apparent that the parameter $r_0$ applies uniformly across all boundary points. It is expected that $r_0$ be small and $M_0$ be large. 
		\item Let $f,\;g \in C^{0,1}(\overline{\mathcal O})$ such that there exists a constant $H_0>0$ with $f-g\geq H_0$
    on $\overline{\mathcal O}$. Then it follows from \cite{henrot2005variation}  that the domain $\Omega_f^g$ given by \eqref{domainrefdef} possesses the Lipschitz regularity depicted in Definition \ref{lipsdef}.
	\end{itemize}
\end{remarks}
Let $\Omega$ be a bounded  Lipschitz domain as in  Definition \ref{lipsdef}, and consider the following  elliptic system
\begin{equation}\label{Nuemannproblem}
	\left\{
	\begin{array}{rrrrr}
		\Delta u &=0\;,&\text{on}\; \; \Omega,\\
		\partial_n u &=v\;,&\quad \text{in}\;\;  \partial\Omega.\\
	\end{array}
	\right.
\end{equation}
In the next  result, we collect some well-known results concerning elliptic systems on Lipschitz domains. For details, we refer to \cite{kenig1994harmonic}; see also \cite[Proof of Lemma 4.5]{alessandrini2003size} and the references therein.

\begin{theorem}\label{usfulequalities}
	Let $\Omega$ be a bounded Lipschitz domain in $\mathbb{R}^{d+1}$ ($d\ge 1 $)  with constants  $r_0$ and $M_0$ according to  Definition \ref{lipsdef}, and let $u\in H^{1}(\Omega)$ be a solution of (\ref{Nuemannproblem}) such that  $v \in L^2(\partial\Omega)$. Then $u\in H^1(\partial \Omega)$ and  there exist two positive constants $C_1$ and $C_2$ that depend only on $r_0$, $M_0$, and $\mu_{d+1} (\Omega)$ such that
	\begin{equation}\label{inequalityimportnat}
		\left\Vert  v\right\Vert_{L^2(\partial\Omega)}\leq C_1  \left\Vert   u \right\Vert_{H^1(\partial\Omega)}\;\; \text{and}\;\; \left\Vert \nabla u \right\Vert_{L^2(\partial\Omega)}\leq C_2  \left\Vert  v \right\Vert_{L^2(\partial\Omega)},\;\;
	\end{equation}
	where $\mu_{d+1} $ denotes the  Lebesgue measure in $\mathbb R ^{d+1}$.
\end{theorem}
The following result concerns the Lipschitz propagation of smallness \cite[Lemma~2.2 and the remark p.~63]{alessandrini2000optimal}. A variant of the theorem employing a different \textit{frequency} characterization of $v$ \eqref{Nuemannproblem} can be found in \cite{alessandrini2003size}.

\begin{theorem}[\cite{alessandrini2000optimal}, Lipschitz propagation of smallness]\label{propagationsmallnes}
	Let $\Omega$ be a bounded Lipschitz domain in $\mathbb{R}^{d+1}$ ($d\ge 1$) with constants $r_0$ and $M_0$ according to  Definition \ref{lipsdef}, and let $u\in H^{1}(\Omega)$ be a solution of (\ref{Nuemannproblem}) with $v\in L^2(\partial \Omega)$. Then, for every $\rho>0$ and for every $(X,y)\in\Omega_{4\rho}$, where 
	$$\Omega_{4\rho}=\{ (X,y)\in \Omega:\;\; d((X,y),\partial \Omega)>4\rho\},$$
	we have
	\begin{equation}\label{equaLipschprop}
		\displaystyle\int_{B^{d+1}_{\rho}((X,y))}|\nabla u|^2\geq C_{\rho} \displaystyle\int_{\Omega} |\nabla u|^2.
	\end{equation}
	Here, $C_{\rho}$ depends only on $\mu_{d+1}(\Omega)$, $r_0$, $M_0$, $\rho$ and $\frac{||v ||_{L^{2}(\partial\Omega)}}{||v ||_{H^{-1/2}(\partial\Omega)}}.$   
\end{theorem}

We end this section by introducing stability results for ill-posed
elliptic Cauchy problems in the case of Lipschitz
domains. 
\begin{theorem} [\cite{choulli2020new}, log-log stability]\label{loglogstability}
	Let $\Omega$ be a bounded Lipschitz domain in  $\mathbb{R}^{d+1}$  ($d\ge 1$) according to  Definition \ref{lipsdef}, 
	let $\Gamma^0$ be a nonempty open subset of $\partial \Omega$ and let $s\in(0,1/2)$. 
	Then there exist some constants $c>e$ and $C>0$ 
	depending on $\Omega$, $\Gamma _0$, and $s$ such that for any $u\in H^2(\Omega)$ with $\Delta u=0$, $u\ne 0$ and 
	$$ \left\Vert u \right\Vert_{L^2(\Gamma^0)}+\left\Vert \nabla u \right\Vert_{L^2(\Gamma^0)}\leq \frac{\left\Vert u \right\Vert_{H^2(\Omega)}}{2\,c},$$ 
	we have 
	$$||u||_{H^1(\Omega)}\leq C \left\Vert u \right\Vert_{H^2(\Omega)} \left[  \ln\ln\left(   \dfrac{\left\Vert u \right\Vert_{H^2(\Omega)}}{||u||_{L^2(\Gamma^0)}+||\nabla u||_{L^2(\Gamma^0)}}     \right)      \right]^{-s/2}.
	$$
\end{theorem}
The above result follows directly from \cite[Corollary 2]{choulli2020new} together with the observation that $$\dfrac{\left\Vert u \right\Vert_{H^2(\Omega)}}{||u||_{L^2(\Gamma^0)}+||\nabla u||_{L^2(\Gamma^0)}}\geq 2 \, c>c.$$

\begin{remark}\label{log stability result}
    For $d=1,2$, a single logarithmic-stability version of Theorem~\ref{loglogstability} holds under the assumption $\Gamma^0 \in C^{1,1}$; see \cite[Theorem~2]{bourgeois2010about}. To the best of our knowledge, these are the only stability estimates presently known for Lipschitz domains. In order to avoid assuming $\Gamma^0\in C^{1,1}$, we rely on Theorem~\ref{loglogstability} throughout this paper.
\end{remark}

\section{ Uniqueness }\label{sec:uniqueness}
In this section, we establish that the operator $\Lambda_{t_0}$ given by  (\ref{operatorofbottom}) is locally one-to-one. It is important to note that the stability analysis in the upcoming section builds on the results established in this part.

By  Remark \ref{Assumption of paper}, the functions $b$, $b_0$, $\zeta$ and $\zeta_0$ are  in $C^1\left(\overline{\mathcal O }\right)$ and $\phi\in H^2(\Omega^{\zeta}_{b})$, $\phi_0\in H^2(\Omega^{\zeta_0}_{b_0})$, where $\phi$ and $\phi_0$  denote the solutions of (\ref{pd3}) and (\ref{pd4}), respectively. To deal with the geometry, we denote  
\begin{equation}\label{defs1ands2}
	S_1=(\zeta-\zeta_0)^{-1}(\mathbb{R}^*_{+}),\;\;S_2=(\zeta-\zeta_0)^{-1}(\mathbb{R}_{-}),
\end{equation}
and we introduce the following Lipschitz continuous functions (see Figure \ref{fig:myplot2})
\begin{equation}\label{defunderlinezetaandb}
	\underline{\zeta}(X) :=\min(\zeta(X),\zeta_0(X)),\;\;\overline{b}(X) :=\max(b(X),b_0(X)),\;\forall X\in \mathcal O .
\end{equation}
\begin{figure}
	\centering
	\begin{tikzpicture}
		\begin{axis}[
			axis lines = middle,
			xlabel = $x$,
			ylabel = $y$,
			xtick =  \empty,
			ytick = \empty,
			xmin = -1,
			xmax = 10,
			ymin = -3, 
			ymax = 15, 
			no markers,
			grid = none,
			every axis x label/.style={at={(current axis.right of origin)},anchor=north east},
			every axis y label/.style={at={(current axis.above origin)},anchor=north east},
			]
			
			\addplot[black, domain=-1:10, samples=100] {sin(deg(x))*cos(deg(x*2))+10} node[pos=0.95, anchor=north] {$\zeta_0 $};
			
			\addplot[dashed, domain=-1:10, samples=100] {2*cos(deg(x*2))+10} node[pos=0.95, anchor=south] {$\zeta$} ;
			\addplot[line width=1.5pt, black, domain=1.5:8.5, samples=100] {min(2*cos(deg(x*2))+10,sin(deg(x))*cos(deg(x*2))+10)} node[pos=0.6, anchor=north] {$\underline{\zeta}$} ;
			
			\addplot[black, domain=-1:10, samples=100, name path=B] {0.5*sin(deg(x))*cos(deg(x))+2.5} node[pos=0.9, anchor=south west] {$b_0 $}; 
			
			\addplot[dashed, domain=-1:10, samples=100, name path=B] {cos(deg(x))+2} node[pos=0.95, anchor=south west] {$b $} ;
			
			\addplot[line width=1.5pt, black, domain=1.5:8.5, samples=100] {max(cos(deg(x))+2,0.5*sin(deg(x))*cos(deg(x))+2.5)} node[pos=0.6, anchor=south west] {$\overline{b} $};

			\draw[black] (axis cs:1.5,0) -- (axis cs:1.5,9.1)  ;
			\draw[black] (axis cs:8.5,0) -- (axis cs:8.5,9.8) ;
			
			\fill (axis cs:1.5,0) circle[radius=2pt] node[below left] {$a_1$};
			\fill (axis cs:8.5,0) circle[radius=2pt] node[below left] {$a_2$};
			
			
			
			
			
		\end{axis}
	\end{tikzpicture}
	\caption{Graphs of the functions $b_0$, $b$, $\zeta _0$ and $\zeta$  for $d=1$ and $\mathcal O=(a_1,a_2)$.}
	\label{fig:myplot2}
\end{figure}
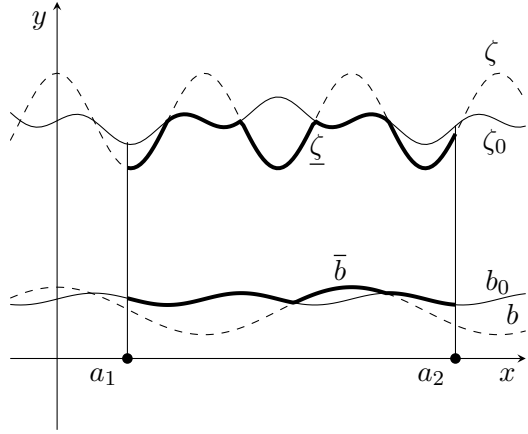

We first state the following inequality, which is central to establishing both uniqueness and stability for our geometric inverse problem.
\begin{lemma}\label{Lemma1useful}
	Let $b$, $b_0$, $\zeta$ and $\zeta_0$ be functions in $C^1\left(\overline{\mathcal O }\right)$ such that (\ref{conditionofliquid}) holds, and  let $\phi\in H^2(\Omega^{\zeta}_{b})$, $\phi_0\in H^2(\Omega^{\zeta_0}_{b_0})$ satisfy systems  (\ref{pd3}) and (\ref{pd4}), respectively. Moreover, assume that
	\begin{equation}\label{conditiononsurfaces}
		|\zeta(X)-\zeta_0(X)|\leq \frac{H_0}{2},\quad \forall X\in \mathcal O .  
	\end{equation}
	Then
	\begin{equation}\label{inequality1}
		\begin{split}
			\displaystyle\int_{\Omega_{b_0}^{\underline{\zeta}}\backslash \Omega_{b}^{\underline{\zeta}}}\left|\nabla\phi_0\right|^2+\displaystyle\int_{\Omega_{b}^{\underline{\zeta}}\backslash \Omega_{b_0}^{\underline{\zeta}}}\left|\nabla\phi\right|^2
			&\leq \displaystyle\int_{\Gamma^{\underline{\zeta}}} \partial_n\phi(\phi-\phi_0)+\displaystyle\int_{\Gamma^{\underline{\zeta}}} \phi_0(\partial_n\phi_0-\partial_n\phi)\\
			&\hspace{0.5cm}+\displaystyle\int_{\Gamma_{\overline{b}}^{\underline{\zeta}}} \partial_n\phi(\phi-\phi_0)+\displaystyle\int_{\Gamma_{\overline{b}}^{\underline{\zeta}}} \phi_0(\partial_n\phi_0-\partial_n\phi)\\
			&\hspace{0.7cm}+\displaystyle\int_{\Gamma_{b}^{\overline{b}}} \partial_n\phi \phi + \displaystyle\int_{\Gamma_{b_0}^{\overline{b}}} \partial_n\phi_0 \phi_0-2\displaystyle\int_{\Gamma^{\overline{b}}} \partial_n\phi \phi_0, 
		\end{split}
	\end{equation}
	where $\underline{\zeta}$ and $\overline{b}$ are given by (\ref{defunderlinezetaandb}). 
\end{lemma}

\begin{proof}
	Fix $X\in \mathcal{O}$. Considering separately the cases $\zeta(X)\ge \zeta_0(X)$ and $\zeta(X)<\zeta_0(X)$, the assumptions \eqref{conditionofliquid} and \eqref{conditiononsurfaces} yield
	\begin{equation}\label{conditionontheinfandsubofthedom}
		\underline{\zeta}(X)-\overline{b}(X)\geq \frac{H_0}{2},\quad \forall X\in \mathcal O .
	\end{equation}
	The inequalities (\ref{conditionofliquid}) and (\ref{conditionontheinfandsubofthedom}) ensure that the domains $\Omega_{\overline{b}}^{  \underline{\zeta}  }$, $\Omega_{b}^{  \underline{\zeta}  }$, and   $\Omega_{b_0}^{  \underline{\zeta}  }$  are well defined (see \eqref{domainrefdef}). Following \cite{kang1997inverse}, we get
	\begin{equation}\label{first-eq-proof-key-lemma}
		\begin{split}
			\displaystyle\int_{\Omega_{\overline{b}}^{  \underline{\zeta}  }}\left| \nabla\phi -\nabla\phi_0\right|^2+\displaystyle\int_{\Omega_{b_0}^{\underline{\zeta}}\backslash \Omega_{b}^{\underline{\zeta}}}\left|\nabla\phi_0\right|^2+\displaystyle\int_{\Omega_{b}^{\underline{\zeta}}\backslash \Omega_{b_0}^{\underline{\zeta}}}\left|\nabla\phi\right|^2
			&=\displaystyle\int_{\Omega_{\overline{b}}^{  \underline{\zeta}  }}\left| \nabla\phi\right|^2+\displaystyle\int_{\Omega_{b}^{\underline{\zeta}}\backslash \Omega_{b_0}^{\underline{\zeta}}}\left|\nabla\phi\right|^2+\displaystyle\int_{\Omega_{\overline{b}}^{  \underline{\zeta}  }}\left| \nabla\phi_0\right|^2\\&\quad+\displaystyle\int_{\Omega_{b_0}^{\underline{\zeta}}\backslash \Omega_{b}^{\underline{\zeta}}}\left|\nabla\phi_0\right|^2-2\displaystyle\int_{\Omega_{\overline{b}}^{\underline{\zeta}}}\nabla\phi\cdot\nabla\phi_0.\\
		\end{split}
	\end{equation}
	Using the fact that 
	$$\Omega_{\overline{b}}^{\underline{\zeta}}\bigcup\left(\Omega_{b}^{\underline{\zeta}}\backslash\Omega_{b_0}^{\underline{\zeta}}\right)=\Omega_{b}^{\underline{\zeta}}\;\;\text{and }\;\;\Omega_{\overline{b}}^{\underline{\zeta}}\bigcup\left(\Omega_{b_0}^{\underline{\zeta}}\backslash\Omega_{b}^{\underline{\zeta}}\right)=\Omega_{b_0}^{\underline{\zeta}},$$
	and observing that the sets $\Omega_{\overline{b}}^{\underline{\zeta}}$, $\Omega_{b}^{\underline{\zeta}}\backslash\Omega_{b_0}^{\underline{\zeta}}$ and $\Omega_{b_0}^{\underline{\zeta}}\backslash\Omega_{b}^{\underline{\zeta}}$ are disjoint (see Figure \ref{fig:myplot2}), we obtain from \eqref{first-eq-proof-key-lemma}
	\begin{equation}\label{lemma1equation1}
		\begin{split}
			\displaystyle\int_{\Omega_{b_0}^{\underline{\zeta}}\backslash \Omega_{b}^{\underline{\zeta}}}\left|\nabla\phi_0\right|^2+\displaystyle\int_{\Omega_{b}^{\underline{\zeta}}\backslash \Omega_{b_0}^{\underline{\zeta}}}\left|\nabla\phi\right|^2&\leq
			\displaystyle\int_{\Omega_{b}^{  \underline{\zeta}  }}\left| \nabla\phi\right|^2+\displaystyle\int_{\Omega_{b_0}^{  \underline{\zeta}  }}\left| \nabla\phi_0\right|^2-2\displaystyle\int_{\Omega_{\overline{b}}^{\underline{\zeta}}}\nabla\phi\cdot\nabla\phi_0.
		\end{split}
	\end{equation}
	On the other hand,  using Remark \ref{remarks-on-the-Lipshitz-domain},  \eqref{conditionontheinfandsubofthedom} with the fact that $\underline\zeta \in C^{0,1}(\overline{\mathcal O })$, and Green's first identity for Lipschitz domains (see \cite{grisvard2011elliptic,necas2011direct}), we obtain
	\begin{equation}\label{lamma1eq2}
		\begin{split}
			\displaystyle\int_{\Omega_{b}^{  \underline{\zeta}  }}\left| \nabla\phi\right|^2&=\displaystyle\int_{\Gamma^{\underline{\zeta}}} \partial_n\phi \phi +\displaystyle\int_{\Gamma_{b}^{\underline{\zeta}}} \partial_n\phi \phi \\
			&=\displaystyle\int_{\Gamma^{\underline{\zeta}}} \partial_n\phi \phi +\displaystyle\int_{\Gamma_{b}^{\overline{b}}} \partial_n\phi \phi+\displaystyle\int_{\Gamma_{\overline{b}}^{\underline{\zeta}}} \partial_n\phi \phi.
		\end{split}
	\end{equation}
	With the same arguments, we obtain
	\begin{equation}\label{lamma1eq3}
		\displaystyle\int_{\Omega_{b_0}^{  \underline{\zeta}  }}\left| \nabla\phi_0\right|^2=\displaystyle\int_{\Gamma^{\underline{\zeta}}} \partial_n\phi_0 \phi_0 +\displaystyle\int_{\Gamma_{b_0}^{\overline{b}}} \partial_n\phi_0 \phi_0+\displaystyle\int_{\Gamma_{\overline{b}}^{\underline{\zeta}}} \partial_n\phi_0 \phi_0,
	\end{equation}
	
	\begin{equation}\label{lamma1eq4}
		\displaystyle\int_{\Omega_{\overline{b}}^{  \underline{\zeta}  }}\nabla\phi\cdot\nabla\phi_0=\displaystyle\int_{\Gamma^{\underline{\zeta}}} \partial_n\phi \phi_0 +\displaystyle\int_{\Gamma^{\overline{b}}} \partial_n\phi \phi_0+\displaystyle\int_{\Gamma_{\overline{b}}^{\underline{\zeta}}} \partial_n\phi \phi_0.
	\end{equation}
	Substituting equations (\ref{lamma1eq2}), (\ref{lamma1eq3}) and (\ref{lamma1eq4}) into (\ref{lemma1equation1}) 
	yields (\ref{inequality1}). The proof is complete.
\end{proof}

In the following result, we bound from above the integrals of the right-hand side of  (\ref{inequality1}) by terms that depend only on measurements taken over $\Gamma^{\underline{\zeta}}$.  The choice of $\Gamma^{\underline{\zeta}}$ is motivated by the fact that, in general,  one of the solutions $\phi$ and $\phi_0$ may not be defined in some regions above $\Gamma^{\underline{\zeta}}$ (see figure \ref{fig:myplot2}).
\begin{proposition}\label{thoeremgeneral}
	 Let $s\in (0,1/2)$. Let $b$, $b_0$, $\zeta$ and $\zeta_0$ be functions in $C^1\left(\overline{\mathcal O }\right)$ such that (\ref{conditionofliquid}) and (\ref{conditiononsurfaces}) hold.  Then there exist two constants $c>e$ and $C>0$ such that if $\phi\in H^2(\Omega^{\zeta}_{b})$ and $\phi_0\in H^2(\Omega^{\zeta_0}_{b_0})$ satisfy systems (\ref{pd3}) and (\ref{pd4}), respectively, with $\phi -\phi _0\ne 0$ and 
	 $$||\phi-\phi_0||_{L^2(\Gamma^{\underline{\zeta}})}+||\nabla (\phi-\phi_0)||_{L^2(\Gamma^{\underline{\zeta}})}\leq\frac{||\phi-\phi_0||_{H^2(\Omega_{\overline{b}}^{\underline{\zeta}})}}{2\,c}, $$
	 we have
	\begin{equation}\label{inequality3}
		\begin{split}
			\displaystyle\int_{\Omega_{b_0}^{\underline{\zeta}}\backslash \Omega_{b}^{\underline{\zeta}}}\left|\nabla\phi_0\right|^2+\displaystyle\int_{\Omega_{b}^{\underline{\zeta}}\backslash \Omega_{b_0}^{\underline{\zeta}}}\left|\nabla\phi\right|^2
			&\leq \displaystyle\int_{\Gamma^{\underline{\zeta}}} \partial_n\phi(\phi-\phi_0)+\displaystyle\int_{\Gamma^{\underline{\zeta}}} \phi_0(\partial_n\phi_0-\partial_n\phi)+Tlog+Tbot,\\
		\end{split}
	\end{equation}
	where  the terms $Tlog$ and $Tbot$ are given by 
	\begin{equation}\label{logterm}
		\begin{split}
			Tlog&:=C \left\Vert \partial_n \phi\right\Vert_{L^{2}(\Gamma_{\overline{b}}^{\underline{\zeta}} )}||\phi-\phi_0||_{H^2(\Omega_{\overline{b}}^{\underline{\zeta}})} \left[  \ln\ln\left(   \dfrac{||\phi-\phi_0||_{H^2(\Omega_{\overline{b}}^{\underline{\zeta}})}}{||\phi-\phi_0||_{L^2(\Gamma^{\underline{\zeta}})}+||\nabla(\phi-\phi_0)||_{L^2(\Gamma^{\underline{\zeta}})}}     \right)      \right]^{-s/2}\\&\hspace{0.6cm}+3C || \phi_0||_{L^2(\partial\Omega_{\overline{b}}^{\underline{\zeta}})}||\phi-\phi_0||_{H^2(\Omega_{\overline{b}}^{\underline{\zeta}})} \left[  \ln\ln\left(   \dfrac{||\phi-\phi_0||_{H^2(\Omega_{\overline{b}}^{\underline{\zeta}})}}{||\phi-\phi_0||_{L^2(\Gamma^{\underline{\zeta}})}+||\nabla (\phi-\phi_0)||_{L^2(\Gamma^{\underline{\zeta}})}}     \right)      \right]^{-s^2/2},
		\end{split}
	\end{equation}
	and
    \begin{equation}\label{bottem}
		\begin{split}
			Tbot&:=\left(||\partial_n \phi||_{L^2(\Gamma_{b}^{\overline{b}})} ||\phi||_{L^{\infty}(\Gamma_{b}^{\underline{\zeta}})}+||\partial_n \phi_0||_{L^2(\Gamma_{b_0}^{\overline{b}})} ||\phi_0||_{L^{\infty}(\Gamma_{b_0}^{\underline{\zeta}})}\right)\left( \mathcal{H}_{d-1}(\partial \mathcal O)  \sup_{X\in \partial \mathcal O} [\overline{b}(X)-b(X)]\right)^{\frac{1}{2}} 
		\end{split}
	\end{equation}
	Here, the constants  $C$ and $c$ depend only on $\Omega_{\overline{b}}^{\underline{\zeta}}$, $\Gamma^{\underline{\zeta}}$ and $s$, and $\mathcal{H}_{d-1}$ denotes the $(d-1)$\nobreakdash-dimensional Hausdorff measure.
\end{proposition}

\begin{proof}
	Let $s\in \left(0,\frac{1}{2} \right)$.   From Lemma \ref{Lemma1useful}, we have 
	\begin{equation}\label{step1inproof}
		\begin{split}
			\displaystyle\int_{\Omega_{b_0}^{\underline{\zeta}}\backslash \Omega_{b}^{\underline{\zeta}}}\left|\nabla\phi_0\right|^2+\displaystyle\int_{\Omega_{b}^{\underline{\zeta}}\backslash \Omega_{b_0}^{\underline{\zeta}}}\left|\nabla\phi\right|^2
			&\leq \displaystyle\int_{\Gamma^{\underline{\zeta}}} \partial_n\phi(\phi-\phi_0)+\displaystyle\int_{\Gamma^{\underline{\zeta}}} \phi_0(\partial_n\phi_0-\partial_n\phi)+J_1+J_2+J_3,
		\end{split}
	\end{equation}
	where
	\begin{equation*}
		\begin{split}
			J_1&=\displaystyle\int_{\Gamma_{\overline{b}}^{\underline{\zeta}}} \partial_n\phi(\phi-\phi_0)+\displaystyle\int_{\Gamma_{\overline{b}}^{\underline{\zeta}}} \phi_0(\partial_n\phi_0-\partial_n\phi),\;\;
			J_2=\displaystyle\int_{\Gamma_{b}^{\overline{b}}} \partial_n\phi \phi + \displaystyle\int_{\Gamma_{b_0}^{\overline{b}}} \partial_n\phi_0 \phi_0,\;\;
			J_3=-2\displaystyle\int_{\Gamma^{\overline{b}}} \partial_n\phi \phi_0.
		\end{split}
	\end{equation*}
	As pointed out earlier, in the subsequent steps of the proof, we aim to derive upper bounds for $J_1$, $J_2$, and  $J_3$, depending solely on measurements over $\Gamma^{\underline{\zeta}}$ and values of $b$ and $b_0$ over $  \partial \mathcal O$. Specifically, we intend to demonstrate that 
	\begin{equation}\label{whatwewanttoproofinthis prop}
		J_1+J_3\leq Tlog \ \ \textrm{ and } \ \ 
		J_2\leq Tbot,
	\end{equation}
	where $Tlog$ and $Tbot$ are given by (\ref{logterm}) and (\ref{bottem}), respectively.
	
	Let us start with the upper bound for $J_1$. Using Cauchy-Schwarz inequality, we obtain 
	\begin{equation*}
		\begin{split}
			J_1&\leq \left\Vert \partial_n \phi\right\Vert_{L^{2}(\Gamma_{\overline{b}}^{\underline{\zeta}} )}\left\Vert  \phi-\phi_0\right\Vert_{L^{2}(\Gamma_{\overline{b}}^{\underline{\zeta}} )}+\left\Vert  \phi_0\right\Vert_{L^{2}(\Gamma_{\overline{b}}^{\underline{\zeta}} )}\left\Vert \partial_n\phi-\partial_n\phi_0 \right\Vert_{L^{2}(\Gamma_{\overline{b}}^{\underline{\zeta}} )}\\
			&\leq \left\Vert \partial_n \phi\right\Vert_{L^{2}(\Gamma_{\overline{b}}^{\underline{\zeta}} )}\left\Vert  \phi-\phi_0\right\Vert_{L^{2}(\partial\Omega_{\overline{b}}^{\underline{\zeta}} )}+\left\Vert  \phi_0\right\Vert_{L^{2}(\Gamma_{\overline{b}}^{\underline{\zeta}} )}\left\Vert \partial_n\phi-\partial_n\phi_0 \right\Vert_{L^{2}(\partial\Omega_{\overline{b}}^{\underline{\zeta}} )}.
		\end{split}
	\end{equation*}
	Using  the trace theorem (see e.g. \cite{grisvard2011elliptic}) and Theorem \ref{usfulequalities},  there exist two positive constants denoted $C_1$ and $C_2$ such that 
	$$J_1\leq C_1 \left\Vert \partial_n \phi\right\Vert_{L^{2}(\Gamma_{\overline{b}}^{\underline{\zeta}} )}\left\Vert  \phi-\phi_0\right\Vert_{H^{1}(\Omega_{\overline{b}}^{\underline{\zeta}} )}+C_2\left\Vert  \phi_0\right\Vert_{L^{2}(\Gamma_{\overline{b}}^{\underline{\zeta}} )}\left\Vert \phi-\phi_0 \right\Vert_{H^{1}(\partial\Omega_{\overline{b}}^{\underline{\zeta}} )}.$$
	Once more, leveraging the trace theorem \cite{ding1996proof}, there exists a positive constant $C_3$ such that 
	$$\left\Vert \phi-\phi_0 \right\Vert_{H^{1}(\partial\Omega_{\overline{b}}^{\underline{\zeta}} )}\leq C_3 \left\Vert \phi-\phi_0 \right\Vert_{H^{\frac{3}{2}+l}(\Omega_{\overline{b}}^{\underline{\zeta}} )}$$
	where $l=\frac{1}{2}-s>0$. By an interpolation argument, we obtain 
	$$\left\Vert \phi-\phi_0 \right\Vert_{H^{1}(\partial\Omega_{\overline{b}}^{\underline{\zeta}} )}\leq C_3 \left\Vert \phi-\phi_0 \right\Vert^s_{H^{1}(\Omega_{\overline{b}}^{\underline{\zeta}} )}\left\Vert \phi-\phi_0 \right\Vert^{1-s}_{H^{2}(\Omega_{\overline{b}}^{\underline{\zeta}} )}.$$
	Then, for  $C_4:=\max(C_1,C_2C_3)$, we infer that
	\begin{equation}\label{I1upperbound}
		J_1\leq C_4 \left\Vert \partial_n \phi\right\Vert_{L^{2}(\Gamma_{\overline{b}}^{\underline{\zeta}} )}\left\Vert  \phi-\phi_0\right\Vert_{H^{1}(\Omega_{\overline{b}}^{\underline{\zeta}} )}+C_4\left\Vert  \phi_0\right\Vert_{L^{2}(\Gamma_{\overline{b}}^{\underline{\zeta}} )}\left\Vert \phi-\phi_0 \right\Vert^{1-s}_{H^{2}(\Omega_{\overline{b}}^{\underline{\zeta}} )}\left\Vert \phi-\phi_0 \right\Vert^s_{H^{1}(\Omega_{\overline{b}}^{\underline{\zeta}} )}.
	\end{equation}
	We turn our attention to $J_3$. Firstly, observe that the boundary  $\Gamma^{\overline{b}}$ can be separated as follows
	$$\Gamma^{\overline{b}}=(\Gamma^{\overline{b}}\cap\Gamma^{b_0})\cup(\Gamma^{\overline{b}}\cap\Gamma^{b}).$$
	Then, using the impermeability condition 
	$$\partial_n\phi_{|_{\Gamma^{b}}}=0\;\;\text{and}\;\; \partial_n\phi_{0|_{\Gamma^{b_0}}}=0,$$
	we derive the following equations
	\begin{equation*}
		\begin{split}
			J_3= -2\displaystyle\int_{\Gamma^{\overline{b}}} \partial_n\phi \phi_0=-2\displaystyle\int_{\Gamma^{\overline{b}}\cap\Gamma^{b_0}} \partial_n\phi \phi_0=-2\displaystyle\int_{\Gamma^{\overline{b}}\cap\Gamma^{b_0}} \left(\partial_n\phi-\partial_n\phi_0\right) \phi_0.
		\end{split}
	\end{equation*}
	Using Cauchy-Schwarz inequality, we obtain 
	$$J_3\leq 2\left\Vert  \phi_0\right\Vert_{L^{2}(\Gamma^{\overline{b}} )} \left\Vert  \partial_n\phi-\partial_n\phi_0\right\Vert_{L^{2}(\partial\Omega_{\overline{b}}^{\underline{\zeta}} )}.$$
	With arguments similar to those used to bound from above $J_1$, one can check that we have 
	
	\begin{equation}\label{upperboundI3}
		J_3\leq 2C_4\left\Vert  \phi_0\right\Vert_{L^{2}(\Gamma^{\overline{b}} )} \left\Vert \phi-\phi_0 \right\Vert^{1-s}_{H^{2}(\Omega_{\overline{b}}^{\underline{\zeta}} )}\left\Vert \phi-\phi_0 \right\Vert^s_{H^{1}(\Omega_{\overline{b}}^{\underline{\zeta}} )}.
	\end{equation}
	Summing up the inequalities (\ref{I1upperbound}) and (\ref{upperboundI3}) term by term, we obtain
	
	\begin{equation}\label{upperboundI23}
		J_1+J_3\leq C_4 \left\Vert \partial_n \phi\right\Vert_{L^{2}(\Gamma_{\overline{b}}^{\underline{\zeta}} )}\left\Vert  \phi-\phi_0\right\Vert_{H^{1}(\Omega_{\overline{b}}^{\underline{\zeta}} )} + 3C_4\left\Vert  \phi_0\right\Vert_{L^{2}(\partial\Omega_{\overline{b}}^{\underline{\zeta}} )} \left\Vert \phi-\phi_0 \right\Vert^{1-s}_{H^{2}(\Omega_{\overline{b}}^{\underline{\zeta}} )}\left\Vert \phi-\phi_0 \right\Vert^s_{H^{1}(\Omega_{\overline{b}}^{\underline{\zeta}} )}.
	\end{equation}
	Applying Theorem \ref{loglogstability} yields the inequality 
	$$J_1+J_3\leq Tlog,$$
	where $Tlog$ is given by (\ref{logterm}). Thus, we have derived the first inequality in (\ref{whatwewanttoproofinthis prop}).
	To complete the proof, it is sufficient to prove that
	$J_2\leq Tbot,$ where  $Tbot$  is given by (\ref{bottem}).
	Using Cauchy-Schwarz inequality, we have 
	\begin{equation}\label{tbotfinalI2}
		\begin{split}
			J_2&=\displaystyle\int_{\Gamma_{b}^{\overline{b}}} \partial_n\phi \phi + \displaystyle\int_{\Gamma_{b_0}^{\overline{b}}} \partial_n\phi_0 \phi_0\\
			&\leq \left\Vert \partial_n \phi\right\Vert_{L^{2}(\Gamma_{b}^{\overline{b}} )} \left\Vert \phi\right\Vert_{L^{2}(\Gamma_{b}^{\overline{b}} )}+\left\Vert \partial_n \phi_0\right\Vert_{L^{2}(\Gamma_{b_0}^{\overline{b}} )} \left\Vert \phi_0\right\Vert_{L^{2}(\Gamma_{b_0}^{\overline{b}} )}. 
		\end{split}
	\end{equation}
	Since $\phi\in H^2({\Omega_{b}^{\underline{\zeta}}})$, by the trace theorem $\phi|_{\Gamma_{b}^{\underline{\zeta}}}\in H^{\frac{3}{2}}(\Gamma_{b}^{\underline{\zeta}})$.
	On the other hand, $\Omega_{b}^{\underline{\zeta}}$ is a subdomain of $\mathbb{R}^{d+1}$, and hence $\Gamma_{b}^{\underline{\zeta}}$ is a $d$-dimensional  manifold, with $d\le 2$. Then  $ H^{\frac{3}{2} }(\Gamma_{b}^{\underline{\zeta}})\hookrightarrow L^\infty (\Gamma_{b}^{\underline{\zeta}})$, which  implies that $\phi|_{\Gamma_{b}^{\underline{\zeta}}}\in L^{\infty}(\Gamma_{b}^{\underline{\zeta}}).$ Hence 
	\begin{equation}\label{tbotfinalI22}
		\begin{split}
			\left\Vert \phi\right\Vert_{L^{2}(\Gamma_{b}^{\overline{b}} )}=\left(\displaystyle\int_{\Gamma_{b}^{\overline{b}}} |\phi|^2\right)^{\frac{1}{2}}&\leq ||\phi||_{L^{\infty}(\Gamma_{b}^{\underline{\zeta}})} \left(\displaystyle\int_{\Gamma_{b}^{\overline{b}}} 1 \right)^{\frac{1}{2}}\\  
			&\leq||\phi||_{L^{\infty}(\Gamma_{b}^{\underline{\zeta}})} \left( \mathcal{H}_{d-1} (\partial \mathcal O)  \sup_{X\in \partial \mathcal O} [\overline{b}(X)-b(X)]\right)^{\frac{1}{2}}.
		\end{split}
	\end{equation}
     Here, since $\partial \mathcal O$ is $C^{1}$, we have $\mathcal H_{d-1}(\partial \mathcal O)<\infty$.\\
	Following  parallel arguments, we can deduce the following inequality
	\begin{equation}\label{tbotfinalI23}
		\begin{split}
			\left\Vert \phi_0\right\Vert_{L^{2}(\Gamma_{b_0}^{\overline{b}} )}\leq ||\phi_0||_{L^{\infty}(\Gamma_{b_0}^{\underline{\zeta}})}
			\left( \mathcal{H}_{d-1} (\partial \mathcal O)  \sup_{X\in \partial \mathcal O} [\overline{b}(X)-b(X)]\right)^{\frac{1}{2}}. 
		\end{split}
	\end{equation}
	Injecting the inequalities (\ref{tbotfinalI22}) and (\ref{tbotfinalI23}) into (\ref{tbotfinalI2}) results in $J_2\leq Tbot$, which completes the proof.
\end{proof}

We now present the first main result of the paper. In particular, this identifiability result extends \cite{fontelos2017bottom} to $C^{1}$ bathymetries and to a truncated-domain setting, while avoiding restrictive assumptions on the boundary conditions.

\begin{theorem}\label{uniqtheorem}
	Let $b$, $b_0$, $\zeta$ and $\zeta_0$ be functions in $C^1\left(\overline{\mathcal O}\right)$ such that (\ref{conditionofliquid}) and (\ref{conditiononsurfaces}) hold, and  let $\phi\in H^2(\Omega^{\zeta}_{b})$, $\phi_0\in H^2(\Omega^{\zeta_0}_{b_0})$ satisfy systems (\ref{pd3}) and (\ref{pd4}), respectively. Assume that
	\begin{equation}\label{conditioninsurface}
		\zeta(X)=\zeta_0(X),\;\; \psi(X)=\psi_0(X),\;\; \partial_n\phi_{|_{\Gamma^{\zeta}}}(X)=\partial_n\phi_{0|_{\Gamma^{\zeta_0}}}(X),\quad \forall X\in \mathcal O,
	\end{equation}
	\begin{equation}
		\label{nonconstant}
		\psi _0\ne Const \textrm { in } { \mathcal O}
	\end{equation}
	and 
	\begin{equation}\label{conditiononbottom}
		b(X)=b_0(X),\quad \forall X\in \partial \mathcal O . 
	\end{equation}
	Then
	$$b(X)=b_0(X),\quad\forall X\in \mathcal O .$$
	
\end{theorem}

\begin{proof}
	Firstly, we  observe that (\ref{conditioninsurface}) implies 
	\begin{equation}\label{condingrad}
		\nabla\phi_{|_{\Gamma^{\zeta}}}=\nabla\phi_{0|_{\Gamma^{\zeta}}}.
	\end{equation}
	Indeed, let us suppose that (\ref{conditioninsurface}) holds. Then, we have 
	$\psi(X)=\psi_0(X)$ and $\zeta(X)=\zeta_0(X)$ for all $X \in \mathcal O $, where, by definition
	$$\psi(X)=\phi(X,\zeta(X)),\;\; \psi_0(X)=\phi_0(X,\zeta_0(X)),\quad\forall X\in \mathcal O .$$
	Combined with $$\partial_n\phi_{|_{\Gamma^{\zeta}}}=\partial_n\phi_{0|_{\Gamma^{\zeta}}},$$ this gives 
	
	$$(\nabla_X\phi)(X,\zeta(X))=(\nabla_X\phi_0)(X,\zeta(X))\;\;\text{and}\;\;(\partial_y\phi)(X,\zeta(X))=(\partial_y\phi_0)(X,\zeta(X)),$$
	i.e. (\ref{condingrad}) holds.  Injecting (\ref{conditioninsurface}), (\ref{conditiononbottom}) and (\ref{condingrad})  into (\ref{inequality3}) results in 
	$$   \displaystyle\int_{\Omega_{b_0}^{\zeta}\backslash \Omega_{b}^{\zeta}}\left|\nabla\phi_0\right|^2+\displaystyle\int_{\Omega_{b}^{\zeta}\backslash \Omega_{b_0}^{\zeta}}\left|\nabla\phi\right|^2=0,$$
	which is equivalent to
	\begin{equation}\label{proofcorebeforlaststep}
		\displaystyle\int_{\Omega_{b_0}^{\zeta}\backslash \Omega_{b}^{\zeta}}\left|\nabla\phi_0\right|^2=0\;\;\text{and}\;\;\displaystyle\int_{\Omega_{b}^{\zeta}\backslash \Omega_{b_0}^{\zeta}}\left|\nabla\phi\right|^2=0.
	\end{equation}
	In the next steps, we aim to prove that (\ref{proofcorebeforlaststep}) implies $b\equiv b_0$ over the set $\mathcal O$. 
	Observe that the set $\Omega_{b_0}^{\zeta}\backslash \Omega_{b}^{\zeta}$ can be written
	$$\Omega_{b_0}^{\zeta}\backslash \Omega_{b}^{\zeta}=\left\{ (X,y)\in {\mathcal O} _+\times\mathbb{R},\;\; b_0(X)<y\leq b(X) \right\},$$
	where ${\mathcal O} _+:=(b-b_0)^{-1}\left(\mathbb{R}_{+}^{*}\right)$.
	Assume that $\mu_{d+1} \left(\Omega_{b_0}^{\zeta}\backslash \Omega_{b}^{\zeta}\right)\neq 0$ ($d=1,2$). Then 
	$\Omega_{b_0}^{\zeta}\backslash \Omega_{b}^{\zeta}\neq \varnothing .$ Let $(X_0,y_0) \in \Omega_{b_0}^{\zeta}\backslash \Omega_{b}^{\zeta}$. 
	By definition,  $b(X_0)>b_0(X_0)$. Using the fact that $b$ and $b_0$ are continuous, there exists a number  $r>0$, such that 
	$$\forall X\in B_r^d(X_0),\quad b(X)>b_0(X).$$
	Now consider the following subset of $\Omega_{b_0}^{\zeta}\backslash \Omega_{b}^{\zeta}$
	$$D:=\left\{ (X,y)\in B_r^d(X_0) \times\mathbb{R},\;\;   b_0(X)<y<b(X) \right\}.$$
	Let $\rho>0$ small enough and pick any  $(X_1,y_1)\in D_{4\rho}$. Then, by using the fact that $b_0$ is a function, we obtain that
	$$(X_1,y_1)\in \big( \Omega_{b_0}^{\zeta}\big) _{4\rho}\;,$$
	where $D_{4\rho}$ and $\big(\Omega_{b_0}^{\zeta}\big)_{4\rho}$ are defined as in Theorem \ref{propagationsmallnes}.\\
	Using  Lipschitz propagation of smallness (\ref{equaLipschprop}), we get the existence of a positive constant $C_{\rho}=C_{\rho}\left(\Omega_{b_0}^{\zeta_0}\right)$ such that 
	\begin{equation}\label{intermproofcor}
		\displaystyle\int_{B_{\rho}^{d+1}((X_1,y_1))}|\nabla \phi_0|^2\geq C_{\rho} \displaystyle\int_{\Omega_{b_0}^{\zeta}} |\nabla \phi_0|^2.
	\end{equation}
	On the other hand
	\begin{equation}\label{stepfinprofcorolorry}
		\displaystyle\int_{B_{\rho}^{d+1}((X_1,y_1))}|\nabla \phi_0|^2\leq \displaystyle\int_{D}|\nabla \phi_0|^2\leq  \displaystyle\int_{\Omega_{b_0}^{\zeta}\backslash \Omega_{b}^{\zeta}}\left|\nabla\phi_0\right|^2.
	\end{equation}
	Combining (\ref{stepfinprofcorolorry}),(\ref{intermproofcor}) and (\ref{proofcorebeforlaststep}) results in the following equation
	$$\displaystyle\int_{\Omega_{b_0}^{\zeta}} |\nabla \phi_0|^2=0.$$
	Then $\nabla \phi_0=0$ a.e. in $\Omega_{b_0}^{\zeta}$ (i.e. the water is at rest), and this gives
	$\phi _0=Const$ in $\Omega_{b_0}^{\zeta} = \Omega_{b_0}^{\zeta _0}$ and $\psi _0=Const$ on $\mathcal O$, contradicting \eqref{nonconstant}. Therefore, we have
	\begin{equation}\label{first0}
		\mu_{d+1} \left(\Omega_{b_0}^{\zeta}\backslash \Omega_{b}^{\zeta}\right)=0. 
	\end{equation}
	Applying similar arguments to $\Omega_{b}^{\zeta}\backslash \Omega_{b_0}^{\zeta}$, we can establish that
	\begin{equation}\label{scond0}
		\mu_{d+1} \left(\Omega_{b}^{\zeta}\backslash \Omega_{b_0}^{\zeta}\right)=0.
	\end{equation}
	Gathering together (\ref{first0}) and (\ref{scond0}), we infer that the bottoms are equal, namely
	$$b(X)=b_0(X),\quad\forall X\in \mathcal O ,$$
	which completes the proof.
\end{proof}

We conclude this section with two remarks: the first concerns an assumption used in the proof of Theorem~\ref{uniqtheorem}, and the second illustrates the possibility of deriving Theorem~\ref{uniqtheorem} from the logarithmic stability in \cite{bourgeois2010about}.

\begin{remark}
 To derive the outcome of Theorem \ref{uniqtheorem}, we assumed that the velocity of the fluid is not constant along the free surface (condition \eqref{nonconstant}). Consequently, we excluded the case of water at rest (still water). However, as mentioned in \cite[Remark 2.12]{fontelos2017bottom}, in this case, it is not possible to uniquely determine the bathymetry.
\end{remark}

\begin{remark}
    If $\zeta,\;\zeta_0 \in C^{1,1}(\mathcal O)$, we can invoke the unique continuation property from \cite{bourgeois2010about} to obtain the above uniqueness result. However, information on the bottom profile at $\partial \mathcal{O} $ is generally required, except in special situations such as those in \cite{fontelos2017bottom,lecaros2020stability}.
\end{remark}

\section{Stability estimates}\label{sec:stability}
In this section, we establish the stability of the considered inverse problem of Section \ref{sec:introinverseproblem}. Based on Proposition \ref{thoeremgeneral}, we aim to bound from above the terms in the right-hand side of (\ref{inequality3}) by measurements on $\Gamma^{\zeta}$ and $\Gamma^{\zeta_0}$ rather than $\Gamma^{\underline{\zeta}}$. Additionally, we need to bound from below the left-hand side of (\ref{inequality3}) by the $L^1$ norm of $b-b_0$.
Following \cite{lecaros2020stability}, we employ the size estimates method. However, instead of assuming a global fatness condition on the inter-bottom region, we introduce a local fatness condition (Hypothesis \ref{hyp-of-paper}) that makes the technique applicable. For readability, the proofs that require direct (but heavy) computations will be shifted to Appendix~\ref{A.1}-\ref{A.2}.

We initiate the upper-bounding procedure with the following lemma, which yields an upper bound for the difference of the traces of the velocity potential on the two surfaces.
\begin{lemma}\label{l1stab}
	Let $b$, $b_0$, $\zeta$ and $\zeta_0$ be  functions in $C^1\left(\overline{\mathcal O }\right)$ such that (\ref{conditionofliquid}) and (\ref{conditiononsurfaces}) hold. Let 
	$\phi\in H ^2(\Omega_b^{\zeta})$ be a solution of system (\ref{pd3}). Then 
	\begin{equation}\label{upperforsurfacephi}
		\begin{split}
			\left\Vert  \psi-\phi_{\zeta_0}   \right\Vert_{L^2(S_1)}&\leq\left\Vert \zeta-\zeta_0  \right\Vert^{\frac{1}{2}}_{L^{\infty}(\mathcal O)}\left\Vert  \phi  \right\Vert_{H ^2(\Omega_b^{\zeta})},\\
			\left\Vert  (\nabla_X\phi )_{\zeta}-(\nabla_X\phi )_{\zeta_0}   \right\Vert_{L^2(S_1)}&\leq\left\Vert \zeta-\zeta_0  \right\Vert^{\frac{1}{2}}_{L^{\infty}(\mathcal O)}\left\Vert  \phi  \right\Vert_{H ^2(\Omega_b^{\zeta})},\\
			\left\Vert  (\partial_y\phi )_{\zeta}-(\partial_y\phi )_{\zeta_0}   \right\Vert_{L^2(S_1)}&\leq\left\Vert \zeta-\zeta_0  \right\Vert^{\frac{1}{2}}_{L^{\infty}(\mathcal O)}\left\Vert  \phi  \right\Vert_{H ^2(\Omega_b^{\zeta})},
		\end{split}
	\end{equation} 
	where $\phi_{\zeta_0}$, $(\nabla_X \phi)_{\zeta _0}$ and $(\partial _y\phi ) _{\zeta _0}$   are defined by 
	\begin{equation}\label{defforsimplication}
		\phi_{\zeta_0}(X)=\phi(X,\zeta_0(X)),\;\; (\nabla_X\phi )_{\zeta _0}(X)= (\nabla_X\phi ) (X,\zeta _0(X)),\;\; (\partial _y\phi )_{\zeta _0}(X)= 
		(\partial _y\phi ) (X,\zeta_0(X)),\;\forall X\in S_1,
	\end{equation}
	$(\nabla_X\phi)_{\zeta }$ and $(\partial _y\phi ) _{\zeta }$ being defined similarly, 
	and the set $S_1$ is given by (\ref{defs1ands2}).
\end{lemma}

\begin{proof}
	Since $\phi \in H ^2(\Omega_b^{\zeta})$, we have $\phi,\,  \partial _y \phi \in L^2(\Omega_b^{\zeta})$, that is 
	$$\displaystyle\int_{\Omega_b^{\zeta}}|\phi|^2<\infty\;\;\text{and}\;\;\displaystyle\int_{\Omega_b^{\zeta}}|\partial_y\phi|^2<\infty.$$
	or, with Fubini theorem
	$$\displaystyle\int_{\mathcal O} \displaystyle\int_{b}^{\zeta} |\phi|^2\, \dd y\,\dd X<\infty\;\;\text{and}\;\;\displaystyle\int_{\mathcal O} \displaystyle\int_{b}^{\zeta}|\partial_y\phi|^2 \dd y\,\dd X <\infty,$$
	thus (see \cite{bogachev2007measure})
	$$ \displaystyle\int_{b}^{\zeta} |\phi|^2 \dd y<\infty\;\;\text{and}\;\; \displaystyle\int_{b}^{\zeta}|\partial_y\phi|^2 \dd y<\infty,\quad \text{ a.e.}\; \text{ on }\; \mathcal O. $$
	Let denote $M$ the set of  $X\in \mathcal O  $ for which at least one of the integrals $\displaystyle\int_{b}^{\zeta} |\phi|^2$ and $\displaystyle\int_{b}^{\zeta}|\partial_y\phi|^2$ is infinite. Then 
	\begin{equation}\label{injectioninR}
		\displaystyle\int_{b(X^{\ast})}^{\zeta(X^{\ast})} |\phi(X^{\ast},y)|^2 \,\dd y<\infty \;\;\text{and}\;\; \displaystyle\int_{b(X^{\ast})}^{\zeta(X^{\ast})}|(\partial_y\phi)(X^{\ast},y)|^2 \,\dd y<\infty,\quad \forall X^{\ast}\in S_1\backslash  M.  
	\end{equation}
	Now, pick any $X^{\ast}$ in $S_1\backslash  M$ and define  $$\phi_{X^{\ast}}(y) :=\phi(X^{\ast},y),\quad y\in ( b(X^{\ast}),\zeta(X^{\ast} ) ) .$$ We infer from (\ref{injectioninR}) that $\phi_{X^{\ast}}\in H^1(b(X^{\ast}),\zeta(X^{\ast}))$. Using Sobolev embedding, we obtain 
	$$\left| \phi_{X^{\ast}}(\zeta(X^{\ast}))-\phi_{X^{\ast}}(\zeta_0(X^{\ast})) \right|=\left|  \displaystyle\int_{\zeta_0(X^{\ast})}^{\zeta(X^{\ast})}\phi_{X^{\ast}}^{'} \,\dd y   \right|.$$
	Using H$\ddot{\text{o}}$lder’s inequality, we derive the following 
	\begin{equation}\label{endoflemma41}
		\left| \phi(X^{\ast},\zeta(X^{\ast}))-\phi(X^{\ast},\zeta_0(X^{\ast})) \right|^2\leq \left| \zeta(X^{\ast}) -\zeta_0(X^{\ast})\right|  \displaystyle\int_{b(X^{\ast})}^{\zeta(X^{\ast})}|(\partial_y\phi)(X^{\ast},\cdot)|^2\, \dd y.
	\end{equation}
	Since the inequality (\ref{endoflemma41}) holds for all $X^{\ast}$ in $S_1\backslash M$, by integrating both sides over $S_1$ and observing that the set $S_1\cap M$ 
	has zero Lebesgue measure, we obtain
	$$\left\Vert  \phi_{\zeta}-\phi_{\zeta_0}   \right\Vert_{L^2(S_1)}\leq\left\Vert \zeta-\zeta_0  \right\Vert^{\frac{1}{2}}_{L^{\infty}(S_1)}\left\Vert  \partial_y\phi  \right\Vert_{L^2(\Omega_b^{\zeta})} .$$
	This establishes the first inequality in (\ref{upperforsurfacephi}).\\
	Replacing $\phi$ with $\nabla_X \phi$ and $\partial_y\phi$, respectively,  and using the fact that $\phi\in H^{2}(\Omega_{b}^{\zeta})$ yields the other required inequalities, which completes the proof.
\end{proof}
Before introducing the proposition that illustrates the upper bounds of the terms in the right-hand side of (\ref{inequality3}), we need the following lemma.
\begin{lemma}\label{uppergradfortheupper}
	Let $b$, $b_0$, $\zeta$ and $\zeta_0$ be functions in $C^1\left(\overline{\mathcal O}\right)$ such that (\ref{conditionofliquid}) and (\ref{conditiononsurfaces}) hold, and  let $\phi\in H^2(\Omega^{\zeta}_{b})$, $\phi_0\in H^2(\Omega^{\zeta_0}_{b_0})$ be solutions of systems (\ref{pd3}) and (\ref{pd4}), respectively. Then
	\begin{equation}\label{relationbetweengradandourmesurement}
		\begin{split}
			\left\Vert (\nabla_X\phi )_{\zeta}-(\nabla_X\phi_0)_{\zeta_0}  \right\Vert_{L^2( \mathcal O )}&\leq \left[G_1\left(\zeta-\zeta_0,\psi-\psi_0,\partial_n\phi_{|_{\Gamma^{\zeta}}}-\partial_n\phi_{0|_{\Gamma^{\zeta_0}}}\right)\right]^{1/2},\\
			\left\Vert (\partial_y\phi ) _{\zeta}- (\partial_y\phi_0)_{ \zeta_0}  \right\Vert_{L^2(\mathcal O )}&\leq \left[G_1\left(\zeta-\zeta_0,\psi-\psi_0,\partial_n\phi_{|_{\Gamma^{\zeta}}}-\partial_n\phi_{0|_{\Gamma^{\zeta_0}}}\right)\right]^{1/2},
		\end{split}
	\end{equation}
	where $(\nabla_X\phi )_{\zeta},  (\partial _y\phi )_{\zeta}, (\nabla_X\phi _0)_{\zeta _0}, 
	(\partial _y\phi_0 )_{\zeta_0}$ are defined according to (\ref{defforsimplication}). The term $G_1$  is given by
	  \begin{equation}\label{firstmajG1}
	  	G_1=\sup(\widehat{G}_1,\widetilde{G}_1),
	  \end{equation} 
	  where, 
	\begin{equation}\label{firstmajG11}
		\begin{split}
			\widehat{G}_1(\zeta-\zeta_0, \psi-\psi_0 ,\partial_n\phi_{|_{\Gamma^{\zeta  }     }}-\partial_n\phi_{0|_{\Gamma^{\zeta_0    }   }} )&= 3\left\Vert \nabla_X\psi-\nabla_X\psi_0  \right\Vert_{L^{2}(\mathcal O )}^2+3\left\Vert  \partial_n\phi_{|_{\Gamma^{\zeta   }    }}-\partial_n\phi_{0|_{\Gamma^{\zeta_0 }      }}\right\Vert_{L^{2}(\mathcal O )}^2\\
		&\hspace{0.3cm}+ 3\left\Vert z_3 \right\Vert_{L^{2}(\mathcal O )}^2 \left\Vert \nabla_X\zeta_0-\nabla_X\zeta   \right\Vert_{L^{\infty}(\mathcal O )}^2,\\
		\end{split}
	\end{equation} 
	and 
	\begin{equation}\label{firstmajG12}
		\begin{split}
			\widetilde{G}_1(\zeta-\zeta_0, \psi-\psi_0 ,\partial_n\phi_{|_{\Gamma^{\zeta  }     }}-\partial_n\phi_{0|_{\Gamma^{\zeta_0    }   }} )&=3\left\Vert \nabla_X\psi-\nabla_X\psi_0  \right\Vert_{L^{2}(\mathcal O )}^2+3\left\Vert \nabla_X\zeta_0 \right\Vert_{L^{\infty}(\mathcal O )}^2\widehat{G}_1\\
			&\hspace{0.3cm}+3 \left\Vert (\partial_y\phi)_{\zeta} \right\Vert_{L^{2}(\mathcal O )}^2 \left\Vert \nabla_X\zeta_0-\nabla_X\zeta   \right\Vert_{L^{\infty}(\mathcal O )}^2.\\
		\end{split}
	\end{equation}
	In  \eqref{firstmajG11}, the coefficient $z_3$ is given  by 
	\begin{equation}\label{z1z2z3}
			z_3=\left| \partial_n\phi_{0|_{\Gamma^{\zeta_0       }}} \right|+ \left|\nabla_X\psi_0  \right|+(1+\left| \nabla_X\zeta_0  \right|)\left|(\partial_y\phi_0)_{\zeta_0}  \right|.
	\end{equation}
\end{lemma}
\begin{proof} The proof results from simple (but cumbersome) computations. It is shifted to Appendix \ref {A.1} for ease of reading.  
\end{proof}
With Lemma~\ref{l1stab} and Lemma~\ref{uppergradfortheupper} at hand, we are now in a position to state the following
\begin{proposition}\label{lemmaaboutupperboundofinte}
	Let $b$, $b_0$, $\zeta$ and $\zeta_0$ be functions in  $C^1\left(\overline{\mathcal O}\right)$ such that (\ref{conditionofliquid}) and (\ref{conditiononsurfaces}) hold, and  let  $\phi\in H^2(\Omega^{\zeta}_{b})$, $\phi_0\in H^2(\Omega^{\zeta_0}_{b_0})$ be solutions of systems (\ref{pd3}) and (\ref{pd4}), respectively. Then
	
	\begin{equation}\label{upperboundsforintegral}
		\begin{array}{ll}
			\vspace{0.2cm}
			\displaystyle\int_{\Gamma^{\underline{\zeta}}} \partial_n\phi(\phi-\phi_0)&\hspace{-0.7cm}\leq G_2\left(\zeta-\zeta_0,\psi-\psi_0\right),\\
			\vspace{0.2cm}
			\displaystyle\int_{\Gamma^{\underline{\zeta}}} \phi_0(\partial_n\phi_0-\partial_n\phi)&\hspace{-0.2cm}\leq G_3\left(\zeta-\zeta_0,\partial_n\phi_{|_{\Gamma^{\zeta}}}-\partial_n\phi_{0|_{\Gamma^{\zeta_0}}}\right),\\
			\vspace{0.2cm}
			||\phi-\phi_0||_{L^2(\Gamma^{\underline{\zeta}})}&\hspace{-1cm}\leq \left[G_4\left(\zeta-\zeta_0,\psi-\psi_0\right)\right]^{1/2},\\
			||\nabla (\phi-\phi_0)||_{L^2(\Gamma^{\underline{\zeta}})}&\hspace{-0.5cm}\leq \left[G_5\left(\zeta-\zeta_0,\psi-\psi_0,\partial_n\phi_{|_{\Gamma^{\zeta}}}-\partial_n\phi_{0|_{\Gamma^{\zeta_0}}}\right)\right]^{1/2},
		\end{array}    
	\end{equation}
	where  $G_2$, $G_3$, $G_4$ and $G_5$ are given by
	\begin{equation}\label{G2to4exprsl}
		\begin{array}{ll}
			G_2(\zeta-\zeta_0,\psi-\psi_0)&\hspace{-3cm} :=\left[\left\Vert   z_4\right\Vert_{L^2(S_1)}\left\Vert\phi\right\Vert_{H^2(\Omega_{b}^{\zeta})} +\left\Vert   z_5\right\Vert_{L^2(\mathcal O)}\left\Vert\phi_0\right\Vert_{H^2(\Omega_{b_0}^{\zeta_0})} \right] \left\Vert \zeta-\zeta_0  \right\Vert_{L^{\infty}(\mathcal O )}^{\frac{1}{2}}\\\vspace{0.3cm}
			&\hspace{0.5cm}+\left[\left\Vert   z_4\right\Vert_{L^2(S_1)} +\left\Vert   z_5\right\Vert_{L^2(\mathcal O)} \right] \left\Vert   \psi-\psi_0 \right\Vert_{L^2(\mathcal O )},\\
			G_3\left(\zeta-\zeta_0,\partial_n\phi_{|_{\Gamma^{\zeta}}}-\partial_n\phi_{0|_{\Gamma^{\zeta_0}}}\right)&\hspace{-1.4cm}:=z_7\left[\left\Vert\phi_{0\zeta_0}\right\Vert_{L^2(\mathcal O)}+\left\Vert\phi_{0\zeta}\right\Vert_{L^2(S_2)}\right] \left\Vert\partial_n\phi_{|_{\Gamma^{\zeta}}}-\partial_n\phi_{0|_{\Gamma^{\zeta_0}}}\right\Vert_{L^2(\mathcal O )}\\
			&\hspace{-0.35cm}+z_7\Bigg\{  \left(  \left\Vert \phi_{0\zeta_0} \right\Vert_{L^2(\mathcal O)}+ \left\Vert  \phi_{0\zeta_0}\nabla_X\zeta_0\right\Vert_{L^2(\mathcal O)} \right)\left\Vert \phi  \right\Vert_{H^2(\Omega_b^{\zeta})}  \\
			&\hspace{-0.3cm}+ \left(  \left\Vert \phi_{0\zeta} \right\Vert_{L^2(S_2)}+ \left\Vert  \phi_{0\zeta} \nabla_X\zeta_0\right\Vert_{L^2(S_2)} \right)\left\Vert \phi_0  \right\Vert_{H^2(\Omega_{b_0}^{\zeta_0})}      \Bigg\}\left\Vert   \zeta-\zeta_0\right\Vert_{L^{\infty}(\mathcal O )}^{\frac{1}{2}}\\\vspace{0.3cm}
			&\hspace{-0.3cm}+z_7z_6\left\Vert\nabla_X\zeta-\nabla_X\zeta_0  \right\Vert_{L^{\infty}(\mathcal O )},\\
		 G_4(\zeta-\zeta_0,\psi-\psi_0)&\hspace{-3cm} :=z_7\bigg[4 \left\Vert \psi-\psi_0 \right\Vert_{L^{2}(\mathcal O )}^2
		+  2 (\left\Vert\phi\right\Vert_{H^2(\Omega_b^{\zeta})}^2+\left\Vert\phi_0\right\Vert_{H^2(\Omega_{b_0}^{\zeta_0})}^2)\left\Vert\zeta-\zeta_0\right\Vert_{L^{\infty}(\mathcal O )} \bigg],  \\   \vspace{0.3cm}
		G_5\left(\zeta-\zeta_0,\psi-\psi_0,\partial_n\phi_{|_{\Gamma^{\zeta}}}-\partial_n\phi_{0|_{\Gamma^{\zeta_0}}}\right)&\hspace{-0.2cm} 
		:=4z_7\left[  \left\Vert  \phi \right\Vert_{H^2(\Omega_{b}^{\zeta})}^2+   \left\Vert  \phi_0 \right\Vert_{H^2(\Omega_{b_0}^{\zeta_0})}^2 \right]\left\Vert   \zeta-\zeta_0\right\Vert_{L^{\infty}(\mathcal O )}  +8z_7 G_1.
		\end{array}
	\end{equation}
	
	Here, the term  $G_1$ is defined in (\ref{firstmajG1}) and  coefficients $z_4$, $z_5$, $z_6$, and $z_7$ are given by 
	\begin{equation}\label{z4toz9}
		\left\{
		\begin{array}{ll}
			z_4 &:= \sqrt{1+|\nabla_X\zeta_0|^2}\;\partial_n\phi_{|_{\Gamma^{\zeta_0}}}, \\
			z_5 &:= \sqrt{1+|\nabla_X\zeta|^2}\;\partial_n\phi_{|_{\Gamma^{\zeta}}},  \\
			z_6&:=  \left\Vert \phi_{0\zeta} \right\Vert_{L^2(S_2)}\left[  \left\Vert  (\nabla_X\phi_0)_{\zeta}\right\Vert_{L^2(S_2)}+ \left\Vert \partial_n\phi_{0|_{\Gamma^{\zeta_0}}} \right\Vert_{L^2(\mathcal O)} \right],\\
			&\hspace{0.5cm}+\left\Vert \phi_{0\zeta_0} \right\Vert_{L^2(\mathcal O)}\left[  \left\Vert  (\nabla_X\phi)_{\zeta}\right\Vert_{L^2(\mathcal O)}+ \left\Vert \partial_n\phi_{|_{\Gamma^{\zeta_0}}} \right\Vert_{L^2(S_1)} \right],\\
			z_7&:=   \max\left\{\left\Vert \sqrt{1+|\nabla_X\zeta|^2} \right\Vert_{L^{\infty}(\mathcal O)},\left\Vert \sqrt{1+|\nabla_X\zeta_0|^2} \right\Vert_{L^{\infty}(\mathcal O)}  \right\}.
		\end{array}
		\right.
	\end{equation}
	
\end{proposition}
\begin{proof} Let us sketch the proof. 
	We transform the surface integrals on the left-hand side of (\ref{upperboundsforintegral}) into Lebesgue integrals over $\mathcal O$. Then, employing arguments similar to those used to prove  Lemma \ref{uppergradfortheupper}, we leverage the inequalities provided by Lemma \ref{l1stab} and Lemma \ref{uppergradfortheupper} to derive the desired upper bounds. For ease of reading, the proof is shifted to Appendix~\ref{A.2}.
\end{proof}

\begin{remark}\label{terms-in-O}
Leveraging Lemma \ref{l1stab}, the terms $\left\Vert \phi_{0\zeta} \right\Vert_{L^2(S_2)}$, $\left\Vert  (\nabla_X\phi_0)_{\zeta}\right\Vert_{L^2(S_2)}$, and $\left\Vert \partial_n\phi_{|_{\Gamma^{\zeta_0}}} \right\Vert_{L^2(S_1)}$  of Proposition \ref{lemmaaboutupperboundofinte} can be bounded from above by the $L^{2}$-norm on $\mathcal O$,  e.g.
    $$\left\Vert \phi_{0\zeta} \right\Vert_{L^2(S_2)} \leq \left\Vert  \phi_0  \right\Vert_{H ^2(\Omega_ {b_0}^{\zeta_0})} \left\Vert \zeta-\zeta_0  \right\Vert^{\frac{1}{2}}_{L^{\infty}(\mathcal O)}+\left\Vert \psi_0 \right\Vert_{L^2(\mathcal O)}.$$
\end{remark} 
Having obtained the upper bounds, we now proceed to the next step of the procedure, precisely to bound from below the left-hand side of \eqref{inequality3}. To avoid imposing restrictive assumptions on the region between the bottoms, such as those in \cite{lecaros2020stability}, we adopt the following relaxed fatness hypothesis. 
\begin{hyp}\label{hyp-of-paper}
    Let $r_0$, and $M_0$ be the Lipschitz constants of both $\Omega^{\underline{\zeta}}_{b}$ and $\Omega^{\underline{\zeta}}_{b_0}$ according to Definition \ref{lipsdef}. Suppose $\Omega_{b_0}^{\underline{\zeta}}\backslash \Omega_{b}^{\underline{\zeta}}$  and $\Omega_{b}^{\underline{\zeta}}\backslash \Omega_{b_0}^{\underline{\zeta}}$ admit disjoint decompositions
 $$\Omega_{b_0}^{\underline{\zeta}}\backslash \Omega_{b}^{\underline{\zeta}}=\bigcup_{n\in \mathcal N}D^n,$$
 $$\Omega_{b}^{\underline{\zeta}}\backslash \Omega_{b_0}^{\underline{\zeta}}=\bigcup_{n\in \mathcal N_0}D^n_0,$$
with the index sets  ${\mathcal N},\; {\mathcal N_0}\subset {\mathbb N}^*$ being finite or countable, such that for each $n$ the associated components $D^n$, $ D^n_0$ are nonempty, connected subsets of $\mathbb{R}^{d+1}$  satisfying the fatness condition; that is, there exists a number
 $\rho_n>0$ depending only on $n$, $r_0$, and $M_0$,  and {\bf not on} $\mu_{d+1}(D^n)$ or $\mu_{d+1}(D_0^n)$, such that
\begin{equation}\label{hypfatness}
\left\{
    \begin{array}{ll}
		\mu_{d+1} (D^n_{\rho_n})\geq\frac{1}{2}\mu_{d+1} (D^n),\\
        \mu_{d+1} (D^n_{0\,\rho_n})\geq\frac{1}{2}\mu_{d+1} (D^n_0).
    \end{array}
    \right.
\end{equation}
\end{hyp}

\begin{remarks}$\,$\\
    \begin{itemize}
    \item 
In the hypothesis above, the region between the bottoms is taken to be a finite or countable union of connected, nonempty components, each satisfying the fatness condition. Since
\begin{equation*}
\Omega_{b_0}^{\underline{\zeta}}\backslash \Omega_{b}^{\underline{\zeta}} = \left\{ (X,y)\in {\mathcal O }_+\times\mathbb{R},\;\; b_0(X)<y\le b(X) \right\}, 
\end{equation*}
where ${\mathcal O } _+:=(b-b_0)^{-1}\left(\mathbb{R}^{*}_{+}\right)$  is an open subset of $\mathbb{R}^d$, there exists a sequence $(\mathcal O _n)_{n\in \mathcal N}$ of disjoint connected open sets such that $${\mathcal O} _+=\bigcup_{n\in \mathcal N} {\mathcal O} _n\subset {\mathcal O},$$
	where ${\mathcal N}\subset {\mathbb N}^*$ is finite or countable. Hence
	$$\Omega_{b_0}^{\underline{\zeta}}\backslash \Omega_{b}^{\underline{\zeta}}=\bigcup_{n\in \mathcal N}D^n,$$
	where
    \begin{equation}\label{Dn}
        D^n=\left\{ (X,y)\in  {\mathcal O} _n\times\mathbb{R},\;\; b_0(X)<y\leq b(X) \right\}.
    \end{equation}
An analogous remark applies as well to $\Omega_{b}^{\underline{\zeta}}\backslash \Omega_{b_0}^{\underline{\zeta}}$.
\item The fatness condition is a classical assumption for the size estimate method \cite{alessandrini2002detecting,beretta2017size,lecaros2020stability,morassi2007size}. However, in our work, we do not require the entire region between the bottoms to satisfy it.  More precisely, we allow $\rho_n$ in \eqref{hypfatness} to vanish asymptotically, i.e.  we may have $\rho_n \to 0$ as $n \to \infty$.

\item For any fixed $n\in {\mathcal N}$, if $\partial D^n \in C^{1,\alpha}$, then by \cite[Lemma 2.8]{rosset1998inverse} the set $D^n$ \eqref{Dn} satisfies the fatness condition \eqref{hypfatness}.

    \end{itemize}
\end{remarks}

Now, we can proceed to state the second main result of the paper.

\begin{theorem}\label{stabilitytheorem}
	 Using the same notations as in Proposition \ref{lemmaaboutupperboundofinte},  let  $s \in (0,1/2)$, let $b$, $b_0$, $\zeta$ and $\zeta_0$ be functions in $C^1\left(\overline{\mathcal O }\right)$ such that (\ref{conditionofliquid}) and (\ref{conditiononsurfaces}) hold, and  let $\phi\in H^2(\Omega^{\zeta}_{b})$, $\phi_0\in H^2(\Omega^{\zeta_0}_{b_0})$ be the solutions of systems (\ref{pd3}) and (\ref{pd4}), respectively, with $\phi-\phi_0\neq 0$. There exist two constants $c> e$ and $C>0$ such that, if 
	$$ \sqrt{G_4}+\sqrt{G_5}\leq\frac{||\phi-\phi_0||_{H^2(\Omega_{\overline{b}}^{\underline{\zeta}})}}{2\,c}, $$
	then we have
	\begin{equation}\label{inequalitytheorems}
		\begin{split}
			C_{bot}\left\Vert b-b_0\right\Vert_{L^1(\mathcal O )}
			&\leq \frac{1}{\min\left\{ \displaystyle\int_{\Omega_b^{\zeta}} |\nabla\phi|^2,\displaystyle\int_{\Omega_{b_0}^{\zeta_0}} |\nabla\phi_0|^2  \right\}}\left(G_2+G_3+Tlog_1+Tbot\right) ,\\
		\end{split}
	\end{equation}
	where $Tbot$ is given by (\ref{bottem}) and   $Tlog_1$  is defined  as
	\begin{equation}\label{logterm1}
		\begin{split}
			Tlog_1&=C \left\Vert \partial_n \phi\right\Vert_{L^{2}(\Gamma_{\overline{b}}^{\underline{\zeta}} )}||\phi-\phi_0||_{H^2(\Omega_{\overline{b}}^{\underline{\zeta}})}  \left[  \ln\ln\left(   \dfrac{||\phi-\phi_0||_{H^2(\Omega_{\overline{b}}^{\underline{\zeta}})} }{\sqrt{G_4}+\sqrt{G_5} }    \right)      \right]^{-s/2}\\&\hspace{0.6cm}+3C || \phi_0||_{L^2(\partial\Omega_{\overline{b}}^{\underline{\zeta}})}||\phi-\phi_0||_{H^2(\Omega_{\overline{b}}^{\underline{\zeta}})} \left[  \ln\ln\left(   \dfrac{||\phi-\phi_0||_{H^2(\Omega_{\overline{b}}^{\underline{\zeta}})} }{\sqrt{G_4}+\sqrt{G_5}}     \right)      \right]^{-s^2/2}.
		\end{split}
	\end{equation}
	Here, the constants  $c$ and $C$ are those in Proposition \ref{thoeremgeneral}, and  $C_{bot}$ depends on the Lipschitz constants of $\Omega_{b}^{\zeta}$ and $\Omega_{b_0}^{\zeta_0}$ according to Definition \ref{lipsdef}, as well as on $\mu_{d+1} \left(\Omega_{b}^{\zeta}\right)$, $\mu_{d+1} \left(\Omega_{b_0}^{\zeta_0}\right)$, 
	$\frac{\left\Vert\partial_n\phi\right\Vert_{L^2(\partial\Omega_{b}^{\zeta})}}{\left\Vert\partial_n\phi\right\Vert_{H^{-1/2}(\partial\Omega_{b}^{\zeta})}}$, 
	and $\frac{\left\Vert\partial_n\phi_0\right\Vert_{L^2(\partial\Omega_{b_0}^{\zeta_0})}}{\left\Vert\partial_n\phi_0\right\Vert_{H^{-1/2}(\partial\Omega_{b_0}^{\zeta_0})}}$.
	
\end{theorem}

\begin{proof}
	Substituting the estimates of Proposition~\ref{lemmaaboutupperboundofinte} into Proposition~\ref{thoeremgeneral}, and using the fact that $(\ln\!\ln x)^{-s_i}$ is decreasing on $(e,+\infty)$ for $s_i=s/2,\,s^2/2$, 
	 we only have to bound from below the term
	
	$$\displaystyle\int_{\Omega_{b_0}^{\underline{\zeta}}\backslash \Omega_{b}^{\underline{\zeta}}}\left|\nabla\phi_0\right|^2+\displaystyle\int_{\Omega_{b}^{\underline{\zeta}}\backslash \Omega_{b_0}^{\underline{\zeta}}}\left|\nabla\phi\right|^2$$
	by the $L^1$ norm of $b-b_0$.\\
	Applying Hypothesis \ref{hyp-of-paper}, we have 
    $$\Omega_{b_0}^{\underline{\zeta}}\backslash \Omega_{b}^{\underline{\zeta}}=\bigcup_{n\in \mathcal N}D^n,$$
    where $D^n$ is given by \eqref{Dn}. Moreover, for each $n\in  \mathcal N $, there exists a number $0<\rho_n=\rho_n(r_0,M_0,n)  \ll H_0$ such that 
	\begin{equation}\label{fatnessn}
		\mu_{d+1} (D^n_{\rho_n})\geq\frac{1}{2}\mu_{d+1} (D^n).
	\end{equation}
	
	On the other hand, note that if $q$ is any closed cube in ${\mathbb R}^{d+1}$ of size $\varepsilon_n= \frac{\rho _n}{2}$ such that $q\cap D^n_{\rho _n}\ne\emptyset$, then 
	$q\subset D^n$. Indeed, if $z\in q\cap D^n_{\rho _n}$ and $z'\in q$,  so that  $\Vert z-z'\Vert \le \sqrt{d+1}\varepsilon _n \le \frac{\sqrt{3}}{2}\rho _n <\rho _n $, we infer from
	$z\in  D^n_{\rho _n}$ that  $z'\in D^n$. 
	Following \cite{alessandrini2000optimal}, we can cover $D^n_{\rho_n}$ by a finite sequence $(q_{n,j})_{j\in J_n}$ of internally non-overlapping closed cubes $q_{n,j}$
	of size $\varepsilon_n=\frac{\rho_n}{2}$, such that 
    \begin{equation}\label{non-overlapping}
        D^n_{\rho_n}\subset Q_n := \bigcup_{j\in J_n}  q_{n,j}\subset D^n.
    \end{equation}
	   Then 
	$$\displaystyle\int_{D^n}|\nabla\phi_0|^2\geq \displaystyle\sum_{j\in J_n}\displaystyle\int_{q_{n,j}}|\nabla\phi_0|^2\geq \frac{\displaystyle\int_{q_{n,0}}|\nabla\phi_0|^2}{\mu_{d+1} (q_{n,0})}\displaystyle\sum_{j\in J_n}\mu_{d+1} (q_{n,j}),$$
	where $q_{n,0}$ is such that 
	$$\min_{j\in J_n}\displaystyle\int_{q_{n,j}}|\nabla\phi_0|^2=\displaystyle\int_{q_{n,0}}|\nabla\phi_0|^2>0.$$
	Therefore
	\begin{equation}\label{proofstabilityfirsttemr}
		\displaystyle\int_{D_n}|\nabla\phi_0|^2\geq \frac{\mu_{d+1} (D^n_{\rho_n})}{\mu_{d+1} (q_{n,0})}\displaystyle\int_{q_{n,0}}|\nabla\phi_0|^2.
	\end{equation}
	Let $(X_{n,0},y_{n,0})$ to be the center of $q_{n,0}$ and consider $\widetilde{\rho}_n=\frac{\varepsilon_n}{2}.$ Observing that 
	$D^n_{\rho_n}\subset \Omega_{b_0,\rho_n}^{\zeta_0}$ (see Figure \ref{fig:myplot2}), we obtain 
	$$(X_{n,0},y_{n,0})\in \Omega_{b_0,\rho_n}^{\zeta_0}=\Omega_{b_0,4\widetilde{\rho}_n}^{\zeta_0}.$$
	Using  Theorem  \ref{propagationsmallnes},  there exists a constant $C(\rho_n)$ such that 
	\begin{equation}\label{prooftheroemstabilitystap2}
		\displaystyle\int_{q_{n,0}}|\nabla\phi_0|^2\geq\displaystyle\int_{B_{\widetilde{\rho}_n}^{d+1}((X_{n,0},y_{n,0}))}|\nabla\phi_0|^2\geq C(\rho_n)\displaystyle\int_{\Omega_{b_0}^{\zeta_0}}|\nabla\phi_0|^2.
	\end{equation}
	Putting  (\ref{proofstabilityfirsttemr}) and (\ref{prooftheroemstabilitystap2}) together, and using (\ref{fatnessn}), we obtain 
	\begin{equation}\label{upperforeachdomain}
		\displaystyle\int_{D_n}|\nabla\phi_0|^2 \geq\frac{2^d C(\rho_n)}{\rho_n^{d+1}}\mu_{d+1} (D^n)\displaystyle\int_{\Omega_{b_0}^{\zeta_0}}|\nabla\phi_0|^2.
	\end{equation}
	To simplify the notations, we denote
	$$\widetilde{C}(\rho_n)=\frac{2^dC(\rho_n)}{\rho_n^{d+1}}.$$
	We will divide the first part of the proof into two cases based on the nature of the sequence  $\left(\widetilde{C}(\rho_n)\right)_n.$
    
	\vspace*{0.25cm}
	\noindent  \textit{\bf  Case $1$.} \vspace*{0.25cm}
	There exists $\widetilde{c}>0$ such that 
	$$\forall n\geq1,\quad \widetilde{C}(\rho_n)\geq \widetilde{c}. $$
	This scenario occurs e.g. when the zero set of $b-b_0$ is finite. 
	Using the additivity of integral and (\ref{upperforeachdomain}), we have 
	\begin{equation*}
		\begin{split}
			\displaystyle\int_{\Omega_{b_0}^{\underline{\zeta}}\backslash \Omega_{b}^{\underline{\zeta}}}\left|\nabla\phi_0\right|^2&=\displaystyle\sum_{
				n\in \mathcal N}\displaystyle\int_{D_n}|\nabla\phi_0|^2\\
			&\geq \displaystyle\sum_{n\in \mathcal N }\widetilde{C}(\rho_n)\mu_{d+1} (D^n)\displaystyle\int_{\Omega_{b_0}^{\zeta_0}}|\nabla\phi_0|^2\\
			&\geq \widetilde{c}\displaystyle\int_{\Omega_{b_0}^{\zeta_0}}|\nabla\phi_0|^2\displaystyle\sum_{n\in \mathcal N} \mu_{d+1} (D^n)\\
			&=\widetilde{c}\displaystyle\int_{\Omega_{b_0}^{\zeta_0}}|\nabla\phi_0|^2\mu_{d+1} (\Omega_{b_0}^{\underline{\zeta}}\backslash\Omega_b^{\underline{\zeta}}).
		\end{split}
	\end{equation*}
	
	\vspace*{0.25cm}
	\noindent  \textit{\bf Case $2$.} \vspace*{0.25cm}
	There exists a subsequence $\left(\widetilde{C}(\rho_{n_k})\right)_k$ of $\left(\widetilde{C}(\rho_{n})\right)_n$ such that 
	$$\widetilde{C}(\rho_{n_k})  \searrow  0\;\;\text{as}\;\; k\to  +\infty.$$
	Since 
	$\min_{1\le n\le n_k} \widetilde{C}(\rho_{n}) \le  \widetilde{C}  (\rho _{n_k}  )  \to 0$  as $k\to \infty $, there exists a positive integer $k$ that depends only on $\rho_1,$ $d$ and $\mu_{d+1}(\Omega_{b_0}^{\zeta_0})$ such that
    $$\min_{1\le n\le n_k} \widetilde{C}(\rho_{n}) \leq \frac{\left(\frac{\rho_1}{2}\right)^{d+1}}{\mu_{d+1}(\Omega_{b_0}^{\zeta_0})}.$$
Observing that $\left(\frac{\rho_1}{2}\right)^{d+1}=\mu_{d+1}(q_{1,j})\leq \mu_{d+1}(D^1)$, where  $(q_{1,j})_{j\in J_1}$ is given in \eqref{non-overlapping}, we obtain 
    \begin{equation*}
        \begin{split}
            \min_{1\le n\le n_k} \widetilde{C}(\rho_{n})&\leq \frac{\mu_{d+1}(D^1)}{\mu_{d+1}(\Omega_{b_0}^{\zeta_0})}\\
            &\leq \frac{\displaystyle\sum_{n=1}^{n_k} \mu_{d+1} (D^n)}{\mu_{d+1} (\Omega_{b_0}^{\underline{\zeta}}\backslash\Omega_b^{\underline{\zeta}})}.
        \end{split}
    \end{equation*}
    Then 
	$$\displaystyle\sum_{n=1}^{n_k} \mu_{d+1} (D^n)\geq \left( \min_{1\le n\le n_k} \widetilde{C}(\rho_{n}) \right)  \mu_{d+1} (\Omega_{b_0}^{\underline{\zeta}}\backslash\Omega_b^{\underline{\zeta}}). $$
	Using arguments as those in Case $1$, we obtain 
	\begin{equation*}
		\begin{split}
			\displaystyle\int_{\Omega_{b_0}^{\underline{\zeta}}\backslash \Omega_{b}^{\underline{\zeta}}}\left|\nabla\phi_0\right|^2&\geq\displaystyle\sum_{n=1}^{n_k}\displaystyle\int_{D_n}|\nabla\phi_0|^2\\
			&\geq \displaystyle\sum_{n=1}^{n_k}\widetilde{C}(\rho_n)\mu_{d+1} (D^n)\displaystyle\int_{\Omega_{b_0}^{\zeta_0}}|\nabla\phi_0|^2\\
			&= \displaystyle\int_{\Omega_{b_0}^{\zeta_0}}|\nabla\phi_0|^2\displaystyle\sum_{n=1}^{n_k}\widetilde{C}(\rho_{n})\mu_{d+1} (D^n)\\
			&\geq \displaystyle\int_{\Omega_{b_0}^{\zeta_0}}|\nabla\phi_0|^2 \big( \min_{1\leq n\leq n_k}\widetilde{C}(\rho_{n}) \big) \displaystyle\sum_{n=1}^{n_k}\mu_{d+1} (D^n)\\
			&\geq \left(\min_{1\leq n\leq n_k}\widetilde{C}(\rho_{n})\right)^2 \displaystyle\int_{\Omega_{b_0}^{\zeta_0}}|\nabla\phi_0|^2\mu_{d+1} (\Omega_{b_0}^{\underline{\zeta}}\backslash\Omega_b^{\underline{\zeta}}).
		\end{split}
	\end{equation*}
	Combining both cases leads to the existence of  a number $C_{bot_0}>0$ such that 
	\begin{equation}\label{lowerforphi0}
		\displaystyle\int_{\Omega_{b_0}^{\underline{\zeta}}\backslash \Omega_{b}^{\underline{\zeta}}}\left|\nabla\phi_0\right|^2\geq C_{bot_0}\displaystyle\int_{\Omega_{b_0}^{\zeta_0}}|\nabla\phi_0|^2\mu_{d+1} (\Omega_{b_0}^{\underline{\zeta}}\backslash\Omega_b^{\underline{\zeta}}).
	\end{equation}
	For the integral of $|\nabla\phi|^2$ over $\Omega_{b}^{\underline{\zeta}}
	\backslash\Omega_{b_0}^{\underline{\zeta}}$, and following the same process, 
      we establish the existence of a constant $C_{bot_1}>0$ such that
	\begin{equation}\label{lowerforphi}
		\displaystyle\int_{\Omega_{b}^{\underline{\zeta}}\backslash \Omega_{b_0}^{\underline{\zeta}}}\left|\nabla\phi\right|^2\geq C_{bot_1}\displaystyle\int_{\Omega_{b}^{\zeta}}|\nabla\phi|^2\mu_{d+1} (\Omega_{b}^{\underline{\zeta}}\backslash\Omega_{b_0}^{\underline{\zeta}}).
	\end{equation}
	Summing  (\ref{lowerforphi0}) and (\ref{lowerforphi}) term by term and setting $C_{bot}:=\min\left(C_{bot_0},C_{bot_1}\right)$, we obtain
	
	$$\displaystyle\int_{\Omega_{b_0}^{\underline{\zeta}}\backslash \Omega_{b}^{\underline{\zeta}}}\left|\nabla\phi_0\right|^2+\displaystyle\int_{\Omega_{b}^{\underline{\zeta}}\backslash \Omega_{b_0}^{\underline{\zeta}}}\left|\nabla\phi\right|^2\geq C_{bot}||b-b_0||_{L^1(\mathcal O )}\min\left\{\displaystyle\int_{\Omega_{b_0}^{\zeta_0}}|\nabla \phi_0|^2,\displaystyle\int_{\Omega_{b}^{\zeta}}|\nabla \phi|^2\right\},$$
	which ends the proof.
\end{proof}

We conclude this section with two remarks.
\begin{remark}
	The proof of Theorem \ref{stabilitytheorem}  shows that the constant $C_{bot}$ depends on the oscillation of the bottoms $b$ and $b_0$. Moreover, the first case in the above proof includes the scenario of a finite number of intersections between the bottoms considered in \cite{lecaros2020stability}. 
\end{remark}
\begin{remark}\label{possiblelogstability}  
Based on \cite[Section 5]{lecaros2020stability} and  Remark \ref{log stability result} we can obtain a Log stability by assuming the existence of an ``important'' open subset of $\Gamma^{\underline{\zeta}}$ of class $C^{1,1}$ and invoking the logarithmic estimate in \cite[Theorem 2]{bourgeois2010about} in place of Theorem \ref{loglogstability}. However, it is important to mention that to achieve this, one needs to transform terms in the stability upper bound from those that converge faster than log to terms of the form log, which will impact the convergence in practice.
\end{remark}

\section{Conclusions, further comments, and future work}\label{sec:conclusions}
Using the general water-waves system in $\mathbb{R} ^{d+1}$, $1\le d\le 2$  \cite{lannes2013water}, we have established the uniqueness and the log-log and log stability estimates for the geometric inverse problem of detecting an arbitrary bounded portion of the bottom through measurements of some functions on the free surface of the fluid   together with the measurement of the bottom at the boundary of this selected portion.

Compared with~\cite{fontelos2017bottom}, we proved uniqueness on a truncated fluid domain and showed that the velocity boundary data is not required for bathymetry detection. In addition, we assumed only $C^{1}$ regularity  for the bathymetry. Compared with  \cite{lecaros2020stability}, our stability results do not assume that the free surfaces coincide at the measurement time $t_{0}$. In addition, we do not require the two bottom profiles to intersect at only finitely many points. In the region between the bottoms, we further relax the classical fatness assumption to a local fatness assumption; see Hypothesis \ref{hyp-of-paper}. This illustrates that the  size estimates method can be  used to detect an infinite, countable union of pairwise disjoint connected subsets, each satisfying the fatness condition.

Our analysis applies to both bounded and unbounded fluid domains. In particular, if the domain is bounded, as in \cite{lecaros2020stability}, or if the underwater geometries are flat outside a compact set, then by Proposition \ref{thoeremgeneral} and the facts that
\(\partial_{ n}\phi\big|_{\Gamma_b^{\zeta}}=0\), \(\partial_{ n}\phi_{0}\big|_{\Gamma_{b_0}^{\zeta_0}}=0\) in the bounded case, and
\(\sup_{X\in \partial \mathcal O}\big[\overline{b}(X)-b(X)\big]=0\) in the flat-outside $\overline{\mathcal{O}}$  case, the term \(T_{\mathrm{bot}}\) given by \eqref{bottem}  vanishes. Consequently, there is no need to know the bottom at the boundary points.

Our results were formulated using measurements over the entire free surface above the selected bounded portion of the bottom. Nevertheless, analogous results could have been derived using data from only a subset of that surface, as noted in Remark~\ref{possiblelogstability}. For brevity, we omitted this analysis.

As future work, we plan to exploit our identifiability result (Theorem \ref{uniqtheorem}) to develop an optimization-based method for bathymetry estimation. In order to address large-scale aquatic domains, we intend to rely on the advanced Isogeometric solvers proposed in \cite{meHendrik,montardini2025isogeometric} to construct an efficient solver for the elliptic potential system \eqref{pd3}.

\section*{Acknowledgements}
 LR was partially supported by ANR (ANR-23-EXMA-0007 and ANR-24-CE40-5470) and COFECUB (COFECUB-20232505780P).

\section*{Appendix}
Recall that throughout this paper we omit the dependence on \(t_0\) (see Remark \ref{notationt0}). We also adopt the notation in \eqref{defforsimplication}. 
\appendix
\section{Proof of Lemma \ref{uppergradfortheupper} }\label{A.1}
From  the definition of the normal derivative and using $\psi(X)=\phi(X,\zeta(X)),\;\; X\in \mathcal O $, we obtain the following system

\begin{equation*}
	\left\{
	\begin{array}{ll}
		\vspace{0.2cm}
		\sqrt{1+|\nabla_X\zeta|^2}\partial_n\phi|_{\Gamma^{\zeta}}=-\nabla_X\zeta \cdot (\nabla_X\phi)_{\zeta}+(\partial_y\phi )_{\zeta},\\
		\nabla_X\psi=(\nabla_X\phi)_{\zeta}+(\partial_y\phi)_{\zeta}\nabla_X\zeta.
	\end{array}
	\right.
\end{equation*}
A straightforward algebraic computation gives
\begin{equation}\label{formulagradiphi}
	\left\{
	\begin{array}{ll}
		\vspace{0.2cm}
		(\partial_y\phi )_{\zeta} =\frac{1}{1+|\nabla_X\zeta|^2}\left(\sqrt{1+|\nabla_X\zeta|^2}\,\partial_n\phi_{|_{\Gamma^{\zeta}}}+\nabla_X\zeta\cdot\nabla_X\psi\right),\\
		(\nabla_X\phi )_{\zeta}=\nabla_X \psi-(\partial_y\phi)_{\zeta}\,\nabla_X\zeta.
	\end{array}
	\right.
\end{equation}
Using similar arguments, we obtain the following formulas for $\phi_0$:
\begin{equation}\label{gradphi0fromula}
	\left\{
	\begin{array}{ll}
		\vspace{0.2cm}
		(\partial_y\phi_0 )_{\zeta_0} =\frac{1}{1+|\nabla_X\zeta_0|^2}\left(\sqrt{1+|\nabla_X\zeta_0|^2}\,\partial_n\phi_{0|_{\Gamma^{\zeta_0}}}+\nabla_X\zeta_0\cdot\nabla_X\psi_0\right),\\
		(\nabla_X\phi_0 )_{\zeta_0}=\nabla_X \psi_0-(\partial_y\phi_0)_{\zeta_0}\,\nabla_X\zeta_0.
	\end{array}
	\right.
\end{equation}
We begin with the second inequality (\ref{relationbetweengradandourmesurement}). From systems \eqref{formulagradiphi}–\eqref{gradphi0fromula}, we have
\begin{equation*}
	\begin{split}
		(\partial_y\phi)_{\zeta}-(\partial_y\phi_0)_{\zeta_0}&=\frac{1}{1+|\nabla_X\zeta|^2}\left(\sqrt{1+|\nabla_X\zeta|^2}\partial_n\phi_{|_{\Gamma^{\zeta
		}}}+\nabla_X\zeta\cdot\nabla_X\psi\right)-\frac{1}{1+|\nabla_X\zeta_0|^2}\bigg(\sqrt{1+|\nabla_X\zeta_0|^2}\partial_n\phi_{0|_{\Gamma^{\zeta_0
		}}}+\nabla_X\zeta_0\cdot\nabla_X\psi_0\bigg)\\
		&=\frac{1}{1+|\nabla_X\zeta|^2}\left(\sqrt{1+|\nabla_X\zeta|^2}\partial_n\phi_{|_{\Gamma^{\zeta
		}}}+\nabla_X\zeta\cdot\nabla_X\psi -\sqrt{1+|\nabla_X\zeta_0|^2}\partial_n\phi_{0|_{\Gamma^{\zeta_0
		}}}-\nabla_X\zeta_0\cdot\nabla_X\psi_0 \right)\\
		&\hspace{0.5cm}+\left(\sqrt{1+|\nabla_X\zeta_0|^2}\partial_n\phi_{0|_{\Gamma^{\zeta_0
		}}}+\nabla_X\zeta_0\cdot\nabla_X\psi_0\right)\left( \frac{1}{1+|\nabla_X\zeta|^2}-\frac{1}{1+|\nabla_X\zeta_0|^2}\right)\\
		&=\frac{1}{1+|\nabla_X\zeta|^2}\bigg[ \sqrt{1+|\nabla_X\zeta|^2}\left(\partial_n\phi_{|_{\Gamma^{\zeta
		}}}-\partial_n\phi_{0|_{\Gamma^{\zeta_0
		}}}\right)+\partial_n\phi_{0|_{\Gamma^{\zeta_0
		}}}\left(\sqrt{1+|\nabla_X\zeta|^2}-\sqrt{1+|\nabla_X\zeta_0|^2}\right) \\
		&\hspace{2cm}+ \nabla_X\zeta\cdot\left(\nabla_X\psi-\nabla_X\psi_0\right)+\nabla_X\psi_0\cdot(\nabla_X\zeta-\nabla_X\zeta_0)             \bigg]\\
		&\hspace{2.1cm}+(1+|\nabla_X\zeta_0|^2)(\partial_y\phi_0)_{\zeta_0}\frac{(\nabla_X\zeta_0-\nabla_X\zeta)\cdot(\nabla_X\zeta+\nabla_X\zeta_0)}{(1+|\nabla_X\zeta|^2)(1+|\nabla_X\zeta_0|^2)}\\
		&=\frac{1}{\sqrt{1+|\nabla_X\zeta|^2}}\left(\partial_n\phi_{|_{\Gamma^{\zeta
		}}}-\partial_n\phi_{0|_{\Gamma^{\zeta_0
		}}}\right)+\frac{1}{1+|\nabla_X\zeta|^2}\partial_n\phi_{0|_{\Gamma^{\zeta_0}}}\frac{(\nabla_X\zeta-\nabla_X\zeta_0)\cdot(\nabla_X\zeta+\nabla_X\zeta_0)}{\sqrt{1+|\nabla_X\zeta|^2}+\sqrt{1+|\nabla_X\zeta_0|^2}}\\
		&\hspace{0.5cm}+\frac{\nabla_X\zeta\cdot(\nabla_X\psi-\nabla_X\psi_0)}{1+|\nabla_X\zeta|^2}+\frac{1}{1+|\nabla_X\zeta|^2}\nabla_X\psi_0\cdot(\nabla_X\zeta-\nabla_X\zeta_0)\\&\hspace{3cm}+\frac{(\nabla_X\zeta+\nabla_X\zeta_0)\cdot(\nabla_X\zeta_0-\nabla_X\zeta)}{1+|\nabla_X\zeta|^2}(\partial_y\phi_0)_{\zeta_0}.
	\end{split}
\end{equation*}
Then
\begin{equation*}
	\begin{split}
		\left| (\partial_y\phi)_{\zeta}-(\partial_y\phi_0)_{\zeta_0}\right|&\leq \left| \partial_n\phi_{|_{\Gamma^{\zeta       }}}-\partial_n\phi_{0|_{\Gamma^{\zeta_0       }}} \right|+\left| \partial_n\phi_{0|_{\Gamma^{\zeta_0       }}} \right|\left|  \nabla_X\zeta_0-\nabla_X\zeta\right|+ \left| \nabla_X\psi-\nabla_X\psi_0 \right|\\
		&\hspace{0.3cm}+ \left|\nabla_X\psi_0  \right|\left|\nabla_X\zeta_0-\nabla_X\zeta  \right|+(1+\left|\nabla_X\zeta_0  \right|)\left|(\partial_y\phi_0)_{\zeta_0}  \right| \left| \nabla_X\zeta_0-\nabla_X\zeta \right|\\
		&= \left| \nabla_X\psi-\nabla_X\psi_0 \right|+ \left| \partial_n\phi_{|_{\Gamma^{\zeta    }   }}-\partial_n\phi_{0|_{\Gamma^{\zeta_0     }  }} \right|\\
		&\hspace{0.3cm}+\bigg[  \left| \partial_n\phi_{0|_{\Gamma^{\zeta_0       }}} \right|+ \left|\nabla_X\psi_0  \right|+(1+\left|\nabla_X\zeta_0  \right|)\left|(\partial_y\phi_0)_{\zeta_0}  \right|\bigg]  \left| \nabla_X\zeta_0-\nabla_X\zeta \right|.
	\end{split}
\end{equation*}
Squaring the inequality, rearranging terms, and  invoking the notation \eqref{z1z2z3}, then integrating over $\mathcal{O}$ yields:

$$\left\Vert (\partial_y\phi)_{\zeta}-(\partial_y\phi_0)_{\zeta_0} \right\Vert_{L^{2}(\mathcal O )}^2\leq \widehat{G}_1(\zeta-\zeta_0, \psi-\psi_0 ,\partial_n\phi_{|_{\Gamma^{\zeta   }    }}-\partial_n\phi_{0|_{\Gamma^{\zeta_0   }    }} ),$$
where $\widehat{G}_1$ is given by \eqref{firstmajG11}.\\
We now proceed to establish the first inequality (\ref{relationbetweengradandourmesurement}). Using \eqref{formulagradiphi}-\eqref{gradphi0fromula}, we have
\begin{equation*}
	\begin{split}
		(\nabla_X \phi)_{\zeta}- (\nabla_X \phi_0)_{\zeta_0}&=\nabla_X \psi-(\partial_y\phi)_{\zeta}\,\nabla_X\zeta-\nabla_X \psi_0+(\partial_y\phi_0)_{\zeta_0}\,\nabla_X\zeta_0\\
		&=\nabla_X \psi-\nabla_X \psi_0-(\partial_y\phi)_{\zeta}\big(\nabla_X \zeta -\nabla_X \zeta_0\big)+\big((\partial_y\phi_0)_{\zeta_0}-(\partial_y\phi)_{\zeta} \big)\nabla_X\zeta_0.
	\end{split}
\end{equation*}
Then 
\begin{equation*}
	\begin{split}
		\left| (\nabla_X \phi)_{\zeta}- (\nabla_X \phi_0)_{\zeta_0}\right|&\leq \left|\nabla_X \psi-\nabla_X \psi_0 \right|+\left|\nabla_X \zeta_0\right| \left|(\partial_y\phi_0)_{\zeta_0}-(\partial_y\phi)_{\zeta} \right|+ \left| (\partial_y\phi)_{\zeta}\right| \left| \nabla_X \zeta -\nabla_X \zeta_0\right|
	\end{split}
\end{equation*}
By a similar argument to the one above, we obtain
$$\left\Vert (\nabla_X\phi)_{\zeta}-(\nabla_X\phi_0)_{\zeta_0} \right\Vert_{L^{2}(\mathcal O )}^2\leq \widetilde{G}_1(\zeta-\zeta_0, \psi-\psi_0 ,\partial_n\phi_{|_{\Gamma^{\zeta   }    }}-\partial_n\phi_{0|_{\Gamma^{\zeta_0   }    }} ),$$
where  $\widetilde{G}_1$ is given by \eqref{firstmajG12}, which completes the proof. \qed
\section{Proof of Proposition \ref{lemmaaboutupperboundofinte} }\label{A.2}
In this part, we also adopt the notation (\ref{z4toz9}). We have
\begin{equation*}
	\begin{split}
		\displaystyle\int_{\Gamma^{\underline{\zeta}}} \partial_n\phi(\phi-\phi_0)&=\displaystyle\int_{\mathcal O }\partial_n\phi_{|_{\Gamma^{\underline{\zeta}}}}(\phi_{|_{\Gamma^{\underline{\zeta}}}}-\phi_{0|_{\Gamma^{\underline{\zeta}}}})\sqrt{1+|\nabla_X\underline{\zeta}|^2}\\
		&=\displaystyle\int_{\mathcal O }\left(-\nabla_X\underline{\zeta}\cdot(\nabla_X\phi)_{\underline{\zeta}}+(\partial_y\phi)_{\underline{\zeta}}\right)(\phi_{\underline{\zeta}}-\phi_{0\underline{\zeta}})\\
		&=\displaystyle\int_{S_1}\left(-\nabla_X\zeta_0\cdot(\nabla_X\phi)_{\zeta_0}+(\partial_y\phi)_{\zeta_0}\right)(\phi_{\zeta_0}-\phi_{0\zeta_0})\\
		&\hspace{0.5cm}+\displaystyle\int_{S_2}\left(-\nabla_X\zeta\cdot(\nabla_X\phi)_{\zeta}+(\partial_y\phi)_{\zeta}\right)(\phi_{\zeta}-\phi_{0\zeta}),
	\end{split}
\end{equation*}
where $S_1$ and $S_2$ are given by (\ref{defs1ands2}). Using H$\ddot{\text{o}}$lder's inequality and Lemma \ref{l1stab}, we have 
\begin{equation*}
	\begin{split}
		\displaystyle\int_{\Gamma^{\underline{\zeta}}} \partial_n\phi(\phi-\phi_0)&=\displaystyle\int_{S_1}z_4(\phi_{\zeta_0}-\phi_{\zeta})+\displaystyle\int_{S_1}z_4(\phi_{\zeta}-\phi_{0\zeta_0})\\
		&\hspace{0.5cm}+\displaystyle\int_{S_2}z_5(\phi_{\zeta}-\phi_{0\zeta_0})+\displaystyle\int_{S_2}z_5(\phi_{0\zeta_0}-\phi_{0\zeta})\\
		&\leq \left\Vert  z_4 \right\Vert_{L^2(S_1)}\left\Vert   \phi_{\zeta_0}-\phi_{\zeta} \right\Vert_{L^2(S_1)}+ \left\Vert   z_4\right\Vert_{L^2(S_1)}  \left\Vert \phi_{\zeta}-\phi_{0\zeta_0}  \right\Vert_{L^2(S_1)}\\
		&\hspace{0.5cm}+ \left\Vert  z_5 \right\Vert_{L^2(S_2)}\left\Vert   \phi_{\zeta}-\phi_{0\zeta_0} \right\Vert_{L^2(S_2)}+ \left\Vert   z_5\right\Vert_{L^2(S_2)}  \left\Vert \phi_{0\zeta_0}-\phi_{0\zeta}  \right\Vert_{L^2(S_2)}\\
		&\leq \left\Vert z_4 \right\Vert_{L^2(S_1)}\left\Vert   \zeta-\zeta_0 \right\Vert_{L^{\infty}(\mathcal O )}^{\frac{1}{2}}\left\Vert\phi\right\Vert_{H^2(\Omega_b^{\zeta})}+ \left\Vert   z_4\right\Vert_{L^2(S_1)}  \left\Vert \psi-\psi_0 \right\Vert_{L^2(S_1)}\\
		&\hspace{0.5cm}+ \left\Vert  z_5 \right\Vert_{L^2(S_2)}\left\Vert   \psi-\psi_0 \right\Vert_{L^2(S_2)}+ \left\Vert   z_5\right\Vert_{L^2(S_2)}  \left\Vert \zeta-\zeta_0  \right\Vert_{L^{\infty}(\mathcal O )}^{\frac{1}{2}}\left\Vert\phi_0\right\Vert_{H^2(\Omega_{b_0}^{\zeta_0})}\\
		&\leq G_2(\zeta-\zeta_0,\psi-\psi_0),
	\end{split}
\end{equation*}
where $G_2$,$z_4$ and $z_5$ are defined in (\ref{G2to4exprsl}) and (\ref{z4toz9}), respectively.\\
Following the same approach, we address the second inequality in (\ref{upperboundsforintegral}). We have
\begin{equation}\label{term2startcomp}
	\begin{split}
		\displaystyle\int_{\Gamma^{\underline{\zeta}}} \phi_0(\partial_n\phi_0-\partial_n\phi)&=\displaystyle\int_{\mathcal O }\phi_{0\underline{\zeta}}(\partial_n\phi_{0|_{\Gamma^{\underline{\zeta}}}}-\partial_n\phi_{|_{\Gamma^{\underline{\zeta}}}})\sqrt{1+|\nabla_X\underline{\zeta}|^2}\\
		&=\displaystyle\int_{S_1}\phi_{0\zeta_0}(\partial_n\phi_{0|_{\Gamma^{\zeta_0}}}-\partial_n\phi_{|_{\Gamma^{\zeta_0}}})\sqrt{1+|\nabla_X\zeta_0|^2}+\displaystyle\int_{S_2}\phi_{0\zeta}(\partial_n\phi_{0|_{\Gamma^{\zeta}}}-\partial_n\phi_{|_{\Gamma^{\zeta}}})\sqrt{1+|\nabla_X\zeta|^2}\\
		&=\displaystyle\int_{S_1}\phi_{0\zeta_0}(\partial_n\phi_{0|_{\Gamma^{\zeta_0}}}-\partial_n\phi_{|_{\Gamma^{\zeta}}})\sqrt{1+|\nabla_X\zeta_0|^2}+\displaystyle\int_{S_1}\phi_{0\zeta_0}(\partial_n\phi_{|_{\Gamma^{\zeta}}}-\partial_n\phi_{|_{\Gamma^{\zeta_0}}})\sqrt{1+|\nabla_X\zeta_0|^2}\\
		&\hspace{0.4cm}+\displaystyle\int_{S_2}\phi_{0\zeta}(\partial_n\phi_{0|_{\Gamma^{\zeta}}}-\partial_n\phi_{0|_{\Gamma^{\zeta_0}}})\sqrt{1+|\nabla_X\zeta|^2}+\displaystyle\int_{S_2}\phi_{0\zeta}(\partial_n\phi_{0|_{\Gamma^{\zeta_0}}}-\partial_n\phi_{|_{\Gamma^{\zeta}}})\sqrt{1+|\nabla_X\zeta|^2}.
	\end{split}
\end{equation}
We denote 
$$J_{S_1}:=\displaystyle\int_{S_1}C_{\zeta_0}\phi_{0\zeta_0}(\partial_n\phi_{|_{\Gamma^{\zeta}}}-\partial_n\phi_{|_{\Gamma^{\zeta_0}}}),\;\;J_{S_2}:=\displaystyle\int_{S_2}C_{\zeta}\phi_{0\zeta}(\partial_n\phi_{0|_{\Gamma^{\zeta}}}-\partial_n\phi_{0|_{\Gamma^{\zeta_0}}}),$$
where
$$C_{\zeta} :=\sqrt{1+|\nabla_X\zeta|^2},\;\; C_{\zeta_0}:=\sqrt{1+|\nabla_X\zeta_0|^2}.$$
Then
\begin{equation*}
	\begin{split}
		J_{S_1}&=\displaystyle\int_{S_1}C_{\zeta_0}\phi_{0\zeta_0}\bigg[\frac{1}{\sqrt{1+|\nabla_X\zeta|^2}}(-\nabla_X\zeta\cdot(\nabla_X\phi)_{\zeta}+(\partial_y\phi)_{\zeta}) -\frac{1}{\sqrt{1+|\nabla_X\zeta_0|^2}}(-\nabla_X\zeta_0\cdot(\nabla_X\phi)_{\zeta_0}+(\partial_y\phi)_{\zeta_0})     \bigg]\\
		&=\displaystyle\int_{S_1}C_{\zeta_0}\phi_{0\zeta_0}\frac{1}{\sqrt{1+|\nabla_X\zeta|^2}}\left[-\nabla_X\zeta\cdot(\nabla_X\phi)_{\zeta}+(\partial_y\phi)_{\zeta}+\nabla_X\zeta_0\cdot(\nabla_X\phi)_{\zeta_0}-(\partial_y\phi)_{\zeta_0}
		\right]\\
		&\hspace{0.5cm}+\displaystyle\int_{S_1}C_{\zeta_0}\phi_{0\zeta_0}\left[-\nabla_X\zeta_0\cdot(\nabla_X\phi)_{\zeta_0}+(\partial_y\phi)_{\zeta_0}\right]\left(\frac{1}{\sqrt{1+|\nabla_X\zeta|^2}} -\frac{1}{\sqrt{1+|\nabla_X\zeta_0|^2}}\right)
		\\
		&=\displaystyle\int_{S_1}C_{\zeta_0}\phi_{0\zeta_0}\frac{1}{\sqrt{1+|\nabla_X\zeta|^2}}\bigg[((\partial_y\phi)_{\zeta}-(\partial_y\phi)_{\zeta_0})+\nabla_X\zeta_0\cdot((\nabla_X\phi)_{\zeta_0}-(\nabla_X\phi)_{\zeta})+(\nabla_X\phi)_{\zeta}\cdot(\nabla_X\zeta_0-\nabla_X\zeta)\bigg]\\
		&\hspace{0.5cm}+\displaystyle\int_{S_1}C_{\zeta_0}\phi_{0\zeta_0}\partial_n\phi_{|_{\Gamma^{\zeta_0}}}\sqrt{1+|\nabla_X\zeta_0|^2}\dfrac{(\nabla_X\zeta_0-\nabla_X\zeta)\cdot(\nabla_X\zeta_0+\nabla_X\zeta)}{\sqrt{1+|\nabla_X\zeta|^2}\sqrt{1+|\nabla_X\zeta_0|^2}\left(\sqrt{1+|\nabla_X\zeta|^2}+\sqrt{1+|\nabla_X\zeta_0|^2}\right)}\\
		&\leq \left\Vert C_{\zeta_0}\phi_{0\zeta_0} \right\Vert_{L^2(S_1)}\left\Vert (\partial_y\phi)_{\zeta_0}-(\partial_y\phi)_{\zeta} \right\Vert_{L^2(S_1)}+\left\Vert C_{\zeta_0}\phi_{0\zeta_0} \nabla_X\zeta_0 \right\Vert_{L^2(S_1)} \left\Vert  (\nabla_X\phi)_{\zeta_0}-(\nabla_X\phi)_{\zeta}\right\Vert_{L^2(S_1)}\\
		&\hspace{0.4cm}+\left\Vert C_{\zeta_0} \phi_{0\zeta_0}(\nabla_X\phi)_{\zeta} \right\Vert_{L^1(S_1)}\left\Vert   \nabla_X\zeta_0-\nabla_X\zeta\right\Vert_{L^{\infty}(S_1)}+ \left\Vert C_{\zeta_0}\phi_{0\zeta_0}\partial_n\phi_{|_{\Gamma^{\zeta_0}}}  \right\Vert_{L^1(S_1)}\left\Vert\nabla_X\zeta_0-\nabla_X\zeta  \right\Vert_{L^{\infty}(S_1)}.\\
	\end{split}
\end{equation*}
Therefore
\begin{equation}\label{defJ1}
	\begin{split}
		J_{S_1}&\leq\left[  \left\Vert C_{\zeta_0}\phi_{0\zeta_0} \right\Vert_{L^2(S_1)}+ \left\Vert  C_{\zeta_0}\phi_{0\zeta_0} \nabla_X\zeta_0\right\Vert_{L^2(S_1)} \right]\left\Vert \phi  \right\Vert_{H^2(\Omega_b^{\zeta})}  \left\Vert   \zeta_0-\zeta\right\Vert_{L^{\infty}(S_1)}^{\frac{1}{2}}\\
		&\hspace{0.4cm}+\left\Vert C_{\zeta_0} \phi_{0\zeta_0} \right\Vert_{L^2(S_1)}\left[  \left\Vert  (\nabla_X\phi)_{\zeta}\right\Vert_{L^2(S_1)}+ \left\Vert \partial_n\phi_{|_{\Gamma^{\zeta_0}}} \right\Vert_{L^2(S_1)} \right]\left\Vert\nabla_X\zeta_0-\nabla_X\zeta  \right\Vert_{L^{\infty}(S_1)}.
	\end{split}
\end{equation}
Replacing $\partial_n\phi$,$\phi_{0\zeta_0}$, $C_{\zeta_0}$ and $S_1$ by $\partial_n\phi_0$,$\phi_{0\zeta}$, $C_{\zeta}$ and $S_2$, respectively, we obtain 
\begin{equation}\label{defJ2}
	\begin{split}
		J_{S_2}&\leq\left[  \left\Vert C_{\zeta}\phi_{0\zeta} \right\Vert_{L^2(S_2)}+ \left\Vert C_{\zeta}  \phi_{0\zeta}\nabla_X\zeta_0\right\Vert_{L^2(S_2)} \right]\left\Vert \phi_0  \right\Vert_{H^2(\Omega_{b_0}^{\zeta_0})}  \left\Vert   \zeta_0-\zeta\right\Vert_{L^{\infty}(S_2)}^{\frac{1}{2}}\\
		&\hspace{0.4cm}+\left\Vert  C_{\zeta}\phi_{0\zeta} \right\Vert_{L^2(S_2)}\left[  \left\Vert  (\nabla_X\phi_0)_{\zeta}\right\Vert_{L^2(S_2)}+ \left\Vert \partial_n\phi_{0|_{\Gamma^{\zeta_0}}} \right\Vert_{L^2(S_2)} \right]\left\Vert\nabla_X\zeta_0-\nabla_X\zeta  \right\Vert_{L^{\infty}(S_2)}.
	\end{split}
\end{equation}
Injecting (\ref{defJ1}) and (\ref{defJ2}) into (\ref{term2startcomp}), we obtain  
\begin{equation*}
	\begin{split}
		\displaystyle\int_{\Gamma^{\underline{\zeta}}} \phi_0(\partial_n\phi_0-\partial_n\phi)&\leq\left[\left\Vert C_{\zeta_0}\phi_{0\zeta_0}\right\Vert_{L^2(S_1)}+\left\Vert C_{\zeta}\phi_{0\zeta}\right\Vert_{L^2(S_2)}\right] \left\Vert\partial_n\phi_{|_{\Gamma^{\zeta}}}-\partial_n\phi_{0|_{\Gamma^{\zeta_0}}}\right\Vert_{L^2(\mathcal O )}\\
		&\hspace{0.5cm}+\bigg[  \left(  \left\Vert C_{\zeta_0}\phi_{0\zeta_0} \right\Vert_{L^2(S_1)}+ \left\Vert  C_{\zeta_0}\phi_{0\zeta_0}\nabla_X\zeta_0\right\Vert_{L^2(S_1)} \right)\left\Vert \phi  \right\Vert_{H^2(\Omega_b^{\zeta})}  \\
		&\hspace{0.6cm}+   \left(  \left\Vert C_{\zeta}\phi_{0\zeta} \right\Vert_{L^2(S_2)}+ \left\Vert C_{\zeta} \phi_{0\zeta}\nabla_X\zeta_0\right\Vert_{L^2(S_2)} \right)\left\Vert \phi_0  \right\Vert_{H^2(\Omega_{b_0}^{\zeta_0})}      \bigg]\left\Vert   \zeta_0-\zeta\right\Vert_{L^{\infty}(\mathcal O )}^{\frac{1}{2}}\\
		&\hspace{0.7cm}+Z_6\left\Vert\nabla_X\zeta_0-\nabla_X\zeta  \right\Vert_{L^{\infty}(\mathcal O )},
	\end{split}
\end{equation*}
where
\begin{equation*}
	\begin{split}
		Z_6&=\left\Vert C_{\zeta}\phi_{0\zeta} \right\Vert_{L^2(S_2)}\left[  \left\Vert  (\nabla_X\phi_0)_{\zeta}\right\Vert_{L^2(S_2)}+ \left\Vert \partial_n\phi_{0|_{\Gamma^{\zeta_0}}} \right\Vert_{L^2(S_2)} \right]\\
		&\hspace{0.5cm}+\left\Vert C_{\zeta_0}\phi_{0\zeta_0} \right\Vert_{L^2(S_1)}\left[  \left\Vert  (\nabla_X\phi)_{\zeta}\right\Vert_{L^2(S_1)}+ \left\Vert \partial_n\phi_{|_{\Gamma^{\zeta_0}}} \right\Vert_{L^2(S_1)} \right].
	\end{split}
\end{equation*}
We derive the desired inequality using the definition of $z_6$ and $z_7$ in (\ref{z4toz9}).\\
Let us address the third inequality in  (\ref{upperboundsforintegral}). Using Lemma \ref{l1stab}, We have
\begin{equation*}
	\begin{split}
		||\phi-\phi_0||_{L^2(\Gamma^{\underline{\zeta}})}^2&=\displaystyle\int_{\Gamma^{\underline{\zeta}}} |\phi-\phi_0|^2\\
		&=\displaystyle\int_{\mathcal O } |\phi_{\underline{\zeta}}-\phi_{0\underline{\zeta}}|^2\sqrt{1+|\nabla_X\underline{\zeta}|^2}\\
		&=\displaystyle\int_{S_1} |\phi_{\zeta_0}-\phi_{0\zeta_0}|^2\sqrt{1+|\nabla_X\zeta_0|^2}+\displaystyle\int_{S_2} |\phi_{\zeta}-\phi_{0\zeta}|^2\sqrt{1+|\nabla_X\zeta|^2}\\
		&=\displaystyle\int_{S_1} |(\phi_{\zeta_0}-\phi_{\zeta})+(\phi_{\zeta}-\phi_{0\zeta_0})|^2\sqrt{1+|\nabla_X\zeta_0|^2}\\
		&\hspace{0.5cm}+\displaystyle\int_{S_2} |(\phi_{\zeta}-\phi_{0\zeta_0})+(\phi_{0\zeta_0}-\phi_{0\zeta})|^2\sqrt{1+|\nabla_X\zeta|^2}\\
		&\leq 2\left\Vert \sqrt{1+|\nabla_X\zeta_0|^2} \right\Vert_{L^{\infty}(S_1)} \left\Vert \phi_{\zeta_0}-\phi_{\zeta} \right\Vert_{L^{2}(S_1)}^2+2\left\Vert \sqrt{1+|\nabla_X\zeta_0|^2} \right\Vert_{L^{\infty}(S_1)} \left\Vert \phi_{\zeta}-\phi_{0\zeta_0} \right\Vert_{L^{2}(S_1)}^2\\
		&\hspace{0.3cm}+ 2\left\Vert \sqrt{1+|\nabla_X\zeta|^2} \right\Vert_{L^{\infty}(S_2)} \left\Vert \phi_{\zeta}-\phi_{0\zeta_0} \right\Vert_{L^{2}(S_2)}^2+2\left\Vert \sqrt{1+|\nabla_X\zeta|^2} \right\Vert_{L^{\infty}(S_2)} \left\Vert \phi_{0\zeta_0}-\phi_{0\zeta} \right\Vert_{L^{2}(S_2)}^2\\
		&\leq z_7\bigg(4\left\Vert \psi-\psi_0 \right\Vert_{L^{2}(\mathcal O )}^2+2\left. (\left\Vert\phi\right\Vert_{H^2(\Omega_b^{\zeta})}^2+\left\Vert\phi_0\right\Vert_{H^2(\Omega_{b_0}^{\zeta_0})}^2)\left\Vert\zeta-\zeta_0\right\Vert_{L^{\infty}(\mathcal O )}   \right.\bigg)\\
	\end{split}
\end{equation*}
The result follows from the definition of $G_4$ given in (\ref{G2to4exprsl}).\\
At this stage, it is clear that we have only used Lemma \ref{l1stab} to derive the above inequalities. However, to establish the last result in (\ref{upperboundsforintegral}), we also require the outcome of Lemma \ref{uppergradfortheupper}.\\
Continuing with the same approach as before, we obtain
\begin{equation*}
	\begin{split}
		||\nabla (\phi-\phi_0)||_{L^2(\Gamma^{\underline{\zeta}})}^2&=\displaystyle\int_{\Gamma^{\underline{\zeta}}}|\nabla\phi_0-\nabla\phi|^2\\
		&=\displaystyle\int_{\mathcal O }\left|\nabla\phi_{0|_{\Gamma^{\underline{\zeta}}}}-\nabla\phi_{|_{\Gamma^{\underline{\zeta}}}}\right|^2\sqrt{1+|\nabla_X\underline{\zeta}|^2}\\
		&=\displaystyle\int_{\mathcal O }\left[\left|(\nabla_X\phi_0)_{\underline{\zeta}}-(\nabla_X\phi)_{\underline{\zeta}}\right|^2+\left((\partial_y\phi_0)_{\underline{\zeta}}-(\partial_y\phi)_{\underline{\zeta}}\right)^2\right]\sqrt{1+|\nabla_X\underline{\zeta}|^2}\\
		&=\displaystyle\int_{S_1}\left[\left|(\nabla_X\phi_0)_{\zeta_0}-(\nabla_X\phi)_{\zeta_0}\right|^2+\left((\partial_y\phi_0)_{\zeta_0}-(\partial_y\phi)_{\zeta_0}\right)^2\right]\sqrt{1+|\nabla_X\zeta_0|^2}\\
		&\hspace{0.3cm}+\displaystyle\int_{S_2}\left[\left|(\nabla_X\phi_0)_{\zeta}-(\nabla_X\phi)_{\zeta}\right|^2+\left((\partial_y\phi_0)_{\zeta}-(\partial_y\phi)_{\zeta}\right)^2\right]\sqrt{1+|\nabla_X\zeta|^2}\\
		&=\displaystyle\int_{S_1}\left|(\nabla_X\phi_0)_{\zeta_0}-(\nabla_X\phi)_{\zeta}+(\nabla_X\phi)_{\zeta}-(\nabla_X\phi)_{\zeta_0}\right|^2\sqrt{1+|\nabla_X\zeta_0|^2}\\
		&\hspace{0.3cm}+\displaystyle\int_{S_1}\left[(\partial_y\phi_0)_{\zeta_0}-(\partial_y\phi)_{\zeta}+(\partial_y\phi)_{\zeta}-(\partial_y\phi)_{\zeta_0}\right]^2\sqrt{1+|\nabla_X\zeta_0|^2}\\
		&\hspace{0.4cm}+\displaystyle\int_{S_2}\left|(\nabla_X\phi_0)_{\zeta}-(\nabla_X\phi_0)_{\zeta_0}+(\nabla_X\phi_0)_{\zeta_0}-(\nabla_X\phi)_{\zeta}\right|^2\sqrt{1+|\nabla_X\zeta|^2}\\
		&\hspace{0.5cm}+\displaystyle\int_{S_2}\left[(\partial_y\phi_0)_{\zeta}-(\partial_y\phi_0)_{\zeta_0}+(\partial_y\phi_0)_{\zeta_0}-(\partial_y\phi)_{\zeta}\right]^2\sqrt{1+|\nabla_X\zeta|^2}.
	\end{split}
\end{equation*}
Expanding the powers in the above equality and employing H$\ddot{\text{o}}$lder inequality leads to a direct application of  Lemma \ref{l1stab} and Lemma \ref{uppergradfortheupper}. After regrouping terms, we derive the following inequality
$$||\nabla (\phi-\phi_0)||_{L^2(\Gamma^{\underline{\zeta}})}^2\leq G_5, $$
where $G_5$ is given in (\ref{G2to4exprsl}).
As one should now be familiar with the techniques, we omit these straightforward computational details.
This completes the proof.\qed


\end{document}